\documentclass[a4paper,12pt]{article}

\usepackage{amsfonts,amsmath,amsthm,mathtools,amssymb}
\usepackage{bbm,tabu,longtable,hyperref}
\usepackage[shortlabels]{enumitem}
\usepackage[margin=1in]{geometry}

\theoremstyle{plain}
\newtheorem{theorem}{Theorem}[section]

\newtheorem{corollary}[theorem]{Corollary}
\newtheorem{lemma}[theorem]{Lemma}
\newtheorem{proposition}[theorem]{Proposition}
\newtheorem{computation}[theorem]{Computation}
\newtheorem*{claim*}{Claim}
\newtheorem*{problem*}{Problem}
\newtheorem*{conjecture*}{Conjecture}

\theoremstyle{definition}
\newtheorem{definition}[theorem]{Definition}
\newtheorem{problem}[theorem]{Problem}

\newtheorem{question}[theorem]{Question}

\newcommand\al{\alpha}
\newcommand\bt{\beta}
\newcommand\gm{\gamma}
\newcommand\dl{\delta}
\newcommand\lm{\lambda}
\newcommand\sg{\sigma}

\newcommand\cF{\mathcal{F}}
\newcommand\cH{\mathcal{H}}
\newcommand\cJ{\mathcal{J}}
\newcommand\cM{\mathcal{M}}

\newcommand\F{\mathbb{F}}
\newcommand\N{\mathbb{N}}
\newcommand\Q{\mathbb{Q}}
\newcommand\R{\mathbb{R}}

\newcommand{\A}{\mathrm{A}}
\newcommand{\B}{\mathrm{B}}
\newcommand{\C}{\mathrm{C}}

\newcommand\la{\langle}
\newcommand\ra{\rangle}
\newcommand\lla{\langle\!\langle}
\newcommand\rra{\rangle\!\rangle}

\DeclareMathOperator{\Aut}{Aut}
\DeclareMathOperator{\Inn}{\mathrm{Inn}}
\DeclareMathOperator{\End}{\mathrm{End}}
\DeclareMathOperator{\Ann}{\mathrm{Ann}}
\DeclareMathOperator{\Miy}{\mathrm{Miy}}

\newcommand\ad{\mathrm{ad}}
\newcommand\one{\mathbbm{1}}
\newcommand{\Id}{\mathrm{Id}}

\usepackage{color}

\newcommand{\ie}{i.e.\ }
\newcommand{\eg}{e.g.\ }

\setlist[enumerate,1]{label={\upshape (\alph*)}}
\setlist[enumerate,2]{label={\upshape (\roman*)}}

\title{Automorphism groups of axial algebras}
\author{I.B.~Gorshkov\footnote{Sobolev Institute of Mathematics, Novosibirsk, Russia, email: ilygor8@gmail.com},
J.~M\textsuperscript{c}Inroy\footnote{Department of Physical, Mathematical and Engineering Sciences, University of Chester, Exton Park, Parkgate Rd, Chester, CH1 4BJ, UK, and School of Mathematics, University of Bristol, Fry Building, Woodland Road, Bristol, BS8 1UG, UK, email: j.mcinroy@chester.ac.uk}, 
T.M.~Mudziiri~Shumba\footnote{Sobolev Institute of Mathematics, Novosibirsk, Russia, email: tendshumba@gmail.com}, \\ and 
S.~Shpectorov\footnote{School of Mathematics, University of Birmingham, Edgbaston, Birmingham, B15 2TT, UK, email: s.shpectorov@bham.ac.uk}}

\begin{document}
\maketitle

\begin{abstract}
Axial algebras are a class of commutative non-associative algebras which have a natural group of automorphisms, called the Miyamoto group.  The motivating example is the Griess algebra which has the Monster sporadic simple group as its Miyamoto group.  Previously, using an expansion algorithm, about 200 examples of axial algebras in the same class as the Griess algebra have been constructed in dimensions up to about 300.  In this list, we see many reoccurring dimensions which suggests that there may be some unexpected isomorphisms.  Such isomorphisms can be found when the full automorphism groups of the algebras are known.  Hence, in this paper, we develop methods for computing the full automorphism groups of axial algebras and apply them to a number of examples of dimensions up to $151$.
\end{abstract}

\section{Introduction}

Axial algebras are a class of commutative non-associative algebras which typically have large finite automorphism groups.  They are generated by special idempotents called axes.  These axes satisfy axioms generalising some key properties observed in the Griess algebra, which has the Monster sporadic simple group as its automorphism group.  However, they encompass a much richer class of algebras than just the Griess algebra, with a variety of different automorphism groups.  Recently, axial behaviour has been observed in algebras arising in various areas of mathematics, including vertex operator algebras, non-linear PDEs \cite{T}, and flows on manifolds \cite{F,F1}.

The main focus in the area of axial algebras is currently on two specific classes of algebras, of Jordan and Monster type, which are characterised by the fusion laws given in Figure \ref{fig:fusion laws}. 
\begin{figure}[h]
\setlength{\tabcolsep}{4pt}
\renewcommand{\arraystretch}{1.4}
\centering
	\begin{minipage}[t]{0.35\linewidth}
		\begin{tabular}{|c||c|c|c|}
        \hline
		$\star$ & $1$ &$0$&$\eta$\\
		\hline  \hline
		$1$&$1$& &$\eta$\\
		\hline
		$0$& &$0$&$\eta$\\
		\hline
		$\eta$&$\eta$&$\eta$&$1,0$\\
        \hline
		\end{tabular}
	\end{minipage}
	\begin{minipage}[t]{0.25\linewidth}
		\begin{tabular}{|c||c|c|c|c|}
        \hline
		$\star$ & $1$ &$0$&$\alpha$& $\beta$\\
		\hline \hline
		$1$&$1$& &$\alpha$& $\beta$\\
		\hline
		$0$& &$0$&$\alpha$& $\beta$\\
		\hline
		$\alpha$&$\alpha$&$\alpha$&$1,0$&$\beta$\\
		\hline
		$\beta$&$\beta$&$\beta$&$\beta$&$1,0,\alpha$\\
        \hline
		\end{tabular}
	\end{minipage}
\caption{The $\cJ(\eta)$ and $\cM(\alpha,\beta)$ fusion laws.}
\label{fig:fusion laws}
\end{figure}

The fusion law limits the possible eigenvalues of the adjoint action of an axis and also restricts the multiplication of eigenvectors.  When a fusion law is $C_2$-graded (as the two above are), such algebras $A$ are inherently related to groups: every axis $a\in A$ defines an involution $\tau_a\in\Aut(A)$, called the \emph{Miyamoto} (or tau) involution. Tau involutions for all axes generate a subgroup of $\Aut(A)$ called the \emph{Miyamoto group} $\Miy(A)$ of $A$. Therefore, axial algebras of Jordan and Monster type always admit substantial automorphism groups.

This is in fact no surprise, because the axioms of axial algebras were modelled after specific 
examples related to groups. Idempotents in Jordan algebras, admitting classical groups and 
$G_2$, satisfy the Peirce decomposition \cite{M}, which is nothing but the fusion law 
$\cJ(\frac{1}{2})$, so algebras of Jordan type are simply a generalisation of Jordan 
algebras.\footnote{Jordan algebras are axial only when they are generated by idempotents. This 
excludes, for example, all nilpotent Jordan algebras.} The Monster fusion law first arose, 
with $\al=\frac{1}{4}$ and $\bt=\frac{1}{32}$, in the Griess algebra, whose automorphism group 
is the Monster sporadic simple group. The work around the Moonshine Conjecture in the last two 
decades of the 20th century led to the introduction of the class of vertex operator algebras 
(VOAs) and construction of the key example of the Moonshine VOA $V^\natural$, whose weight-2 
component is the Griess algebra. Further work on VOAs by Miyamoto \cite{Mi} led to the concept 
of an Ising vector in the weight-2 component of an OZ VOA and the corresponding tau 
involution. The attempt by Miyamoto and, especially, Sakuma \cite{Mi1,S} to classify VOAs 
generated by two Ising vectors led Ivanov \cite{I} to introduce the class of Majorana 
algebras, which are the direct predecessors of axial algebras and, specifically, algebras of 
Monster type. Note that algebras of Jordan type are a subclass of the algebras of Monster type 
arising when one of the higher (\ie $\al$, or $\bt$) eigenvalues disappears. For a survey on 
algebras of Jordan and Monster type see \cite{MS1}.

While the smaller class of algebras of Jordan type is close to being fully understood and 
classified, we are far from that for algebras of Monster type. Much of the recent work 
focussed on finding new examples of such algebras. This includes building examples 
for specific groups by computer. Seress \cite{Se} developed an algorithm in the computer 
algebra system GAP \cite{GAP} for computing Majorana algebras for concrete small groups. 
He was able to construct algebras for quite a few groups, the largest of which was the algebra 
of dimension 286 for the Mathieu sporadic simple group $M_{11}$. The limitations of this 
program were that it only worked with $2$-closed algebras and utilised the full set of axioms 
of Majorana algebras. After Seress' passing, his program was restored and improved by Pfeiffer 
and Whybrow \cite{PW}. In particular, the program can now handle some $3$-closed algebras, too.

Another approach, based on expansions, was pioneered by Shpectorov and implemented in GAP by 
Rehren. This was developed significantly by M\textsuperscript{c}Inroy and Shpectorov \cite{axialconstruction, 
MS} and ported into MAGMA \cite{ParAxlAlg, MAGMA}. The expansion program does not require the 
algebra to be $n$-closed for any $n$ and it does not use any of the axioms specific for 
Majorana algebras. In total, we currently have approximately $200$ individual algebras 
computed in the most interesting case of Monster type $(\frac{1}{4},\frac{1}{32})$, where the 
Griess algebra occurs. These algebras are classified in terms of their Miyamoto group, but 
also in terms of their axets (roughly speaking the set of axes) and shapes (see Section \ref{background} for the precise definitions).

This extensive zoo of examples is the starting point of this paper. We want to achieve a 
better understanding of them and find patterns that suggest where individual examples may 
be members of infinite series. In particular, looking at the tables of examples, one 
immediately notices that the dimensions of the algebras tend to repeat, which suggests that 
some of the examples may be isomorphic. Note that there currently are no available practical 
methods of finding such isomorphisms, which motivates the following. 

\begin{problem}
Develop computationally efficient tools to determine whether two axial algebras are isomorphic.
\end{problem}

The same algebra may indeed arise several times with different axets and Miyamoto groups. In 
such a case, there would be one largest realisation of the algebra, involving all possible 
axes in it and the full automorphism group. Note that this full realisation may be unknown for 
the algebra, especially if it is big. So we have the following natural question.

%it is a natural 
%problem to find the full sets of axes in our algebras (equivalently, find the full 
%automorphism groups) and study the list for possible isomorphisms and inclusions among the 
%examples.

\begin{problem}
Find the full automorphism group of an axial algebra.
\end{problem}

We note that the general problem of finding the automorphism group of an algebra or finding 
an isomorphism between two algebras is very difficult.  In particular, there is no reason a 
priori, why the automorphism group cannot be infinite.  To alleviate this in the axial case, 
we show the following.

\begin{theorem}
Let $A$ be a finite-dimensional axial algebra over a field $\F$ of characteristic not two with fusion law $\cF \subseteq \F$.  If $\frac{1}{2} \notin \cF$, then $\Aut(A)$ is finite.
\end{theorem}

It is known that some Jordan algebras, which are axial algebras of Jordan type $\frac{1}{2}$, have infinite automorphism groups.  However, if $\frac{1}{2}$ is in $\cF$, then the automorphism group could still be finite and we give an example of such an algebra in Section \ref{example}.  We also develop further methods in Section \ref{Finiteness} for showing that $\Aut(A)$ is finite even if $\frac{1}{2} \in \cF$ and pose some open questions.

Even if the automorphism group is finite, finding it in general is a difficult problem.  Since 
an axial algebra is generated by axes, it follows that the automorphism group must act 
faithfully on the full set of axes.  So to find the automorphism group, we can instead try to 
find the full set of axes.

\begin{problem}
Find the full set of axes for an axial algebra.
\end{problem}

As axes are idempotents with special properties, one can try to find idempotents in a na\"ive 
way by solving a system of quadratic equations, often using Gr\"obner basis methods.  For a 
finite dimensional algebra, by B\'ezout's Theorem, one should expect there to be $2^{\dim(A)}$ 
solutions, or even infinitely many, when the idempotent variety is of positive dimension.  So, 
this approach may only work for low-dimensional algebra, up to the dimension of about 10.  
Furthermore, once the idempotents have been found, there is a further computation to determine 
which of them are, in fact, axes.  

In Section \ref{eigenvalue-0}, we reduce the problem of finding all axes in an axial algebra 
$A$ to a calculation in the $0$-eigenspace $A_0(a)$ of a known axis $a$.  This $A_0(a)$ is a 
subalgebra and it typically has a significantly smaller dimension. Using the classification of 
$2$-generated axial algebras of Monster type $\left(\frac{1}{4},\frac{1}{32}\right)$, we 
describe a more nuanced algorithm which works well for algebras up to dimension $35$ without 
any need for human input.

We then further generalise the idea of reducing to a smaller subalgebra. Instead of searching 
in the $0$-eigenspace $A_0(a)$ with respect to a single axis, we now do it relative to a set 
$Y=\{a_1,a_2,\ldots a_k\}$ of axes such that each $\tau_i=\tau_{a_i}$ fixes all axes in $Y$. 
Consequently, all involutions $\tau_i$ commute, and so $E=\la \tau_{a_i}\ra_{i\in 
\{1,2,\ldots,k\}}$ is an elementary abelian $2$-subgroup of $\Miy(A)$. For a tuple $(\lm_1,\lm_2,\ldots,\lm_k)\in\cF^k$ we define the components 
$A_{(\lm_1,\lm_2,\ldots,\lm_k)}(Y):=\bigcap_{i=1}^k A_{\lm_i}(a_i)$. If the fusion law $\cF$ 
is Seress, then the joint $0$-eigenspace $U:=A_{(0,0,\ldots ,0)}(Y)$ is a subalgebra, and 
every other component is a $U$-module.

Once we have found the automorphism group of this even smaller subalgebra $U$, we can try to 
extend automorphisms of $U$ to other components $A_{(\lm_1,\lm_2,\ldots,\lm_k)}(Y)$. Each of 
these calculations is now a system of linear, rather than quadratic equations and hence easily 
solvable.  Typically, if the component is small, the solution is a choice of scalar, however 
the solution set may also have higher dimension. By comparing the extensions for different 
components, we find further restrictions which can be used to find the exact scalars and hence 
find the full extension to $\Aut(A)$, if it exists.

We started this project with a very limited goal of finding the automorphism groups of the 
algebras for the symmetric group $S_4$ in dimensions from $6$ to $25$.  These can be computed 
using our nuanced algorithm. Using the more powerful component method and combining computer 
calculations with group-theoretic proofs, we can complete much larger examples, 
the largest currently being an algebra of dimension $151$ for the group $2\times\Aut(A_6)$. 
We expect that eventually, after additional work, we will be able to finish all known 
algebras including the $M_{11}$ example of dimension 
$286$. We summarise the current results in Table \ref{results}. Each algebra is identified by 
the initial Miyamoto group $G$, the initial axet, dimension, and the shape. We follow 
\cite{axialconstruction} concerning the shape notation. Where we found new axes, they are 
shown in the next column, organised in $G$-orbits. The final column contains the full 
automorphism group of $A$. As you can see, we found quite a few cases where the full 
automorphism group is bigger, although in all cases it is still a natural extension of the 
initial group $G$.

\begin{table}[ht!]
\centering
\renewcommand{\arraystretch}{1.5}
\begin{tabular}{cccccc}\hline\hline
$G$&axet&dim&shape&new&$\mathrm{Aut}(A)$\\
\hline\hline
$S_4$&6&13&3A2B&&$S_4$\\
&6&6&3C2B&&$S_4$\\
&$3+6$&23&4A3A2A&$1+6$&$C_2\times S_4$\\
&$3+6$&25&4A3A2B&&$S_4$\\
&$3+6$&12&4A3C2B&&$S_4$\\
&$3+6$&13&4B3A2A&&$S_4$\\
&$3+6$&16&4B3A2B&3&$S_4$\\
&$3+6$&9&4B3C2A&&$S_4$\\
&$3+6$&12&4B3C2B&3&$S_4$\\
\hline
$A_5$&15&26&3A2A&&$S_5$\\
&15&46&3A2B&&$S_5$\\
&15&20&3C2A&&$S_5$\\
&15&21&3C2B&&$S_5$\\
\hline
$S_5$&$10$&10&3C2B&&$S_5$\\
&$10+15$&61&4A&$1+10$&$C_2\times S_5$\\
&$10+15$&36&4B&&$S_5$\\
\hline
$L_3(2)$&21&57&4A3C&&$\Aut(L_3(2))$\\
&21&49&4B3A&&$\Aut(L_3(2))$\\
&21&21&4B3C&&$\Aut(L_3(2))$\\
\hline
$A_6$&45&121&4A3A3A&&$\Aut(A_6)$\\
\hline
$S_6$&15&15&3C2B&&$S_6$\\
&$15+45$&151&4A4A3A2A&$1+15+15+15$&$C_2\times \Aut(A_6)$\\
\hline
$S_3\times S_3$&$3+3+9$&18&3A2A&&$S_3\wr C_2$\\
&$3+3+9$&25&3A2B&$1+3+3$&$C_2\times (S_3\wr C_2)$\\
\hline
$(S_4\times S_3)^+$&$3+18$&24&3C3C3C2A&&$C_3:S_3: S_4$\\
&$3+18$&27&3C3C3C2B&3&$C_3:S_3:S_4$\\
\hline
$S_7$&21&21&3C2B&&$S_7$\\
\hline\hline
\end{tabular}
\caption{Results}\label{results}
\end{table}

In exploring some of the examples given in Table \ref{results}, we notice that some 
algebras contain \emph{twin} axes, which are axes $a$ and $b$ such that $\tau_a=\tau_b$. We 
show how these can be found in a computationally efficient way for algebras of Monster type 
$(\al,\bt)$ by solving mostly linear equations.

Closely related to the notion of twins is the existence of Jordan axes. These are the axes 
which satisfy the tighter Jordan fusion law $\mathcal{J}(\al)$, that is, their 
$\bt$-eigenspaces are trivial. As a result, their tau involutions are trivial. However, as 
the Jordan type fusion law is also graded, there is a different involution, called a 
\emph{sigma involution}, associated to each Jordan axis. We show that, 
at least in the special case $(\al,\bt)=\left(\frac{1}{4},\frac{1}{32}\right)$, all Jordan 
axes are contained in the fixed subalgebra $A_G$, where $G$ is the Miyamoto group of $A$. 
Clearly, this subalgebra is typically very small, and so all Jordan axes can be found in an 
efficient way. The action of the sigma involution of a Jordan axis naturally leads to the 
existence of twins in the algebra. In all the cases where we found Jordan axes, the algebra 
has a unique Jordan axis.

\begin{problem}
Do there exist axial algebras of Monster type with more than one Jordan axis?
\end{problem}

Such examples can be easily constructed in direct sums of algebras, so we are of course only 
interested in indecomposible (\ie connected) examples.  Additionally, algebras of Jordan type are also algebras of Monster type, so we also ignore these.

The paper is organised as follows. In Section \ref{background}, we provide the necessary 
background information about axial algebras, including the key definitions and some of the 
relevant results.  Finiteness of the automorphism group is discussed in Section 
\ref{Finiteness} and, in Section \ref{example}, we give an example where $\al=\frac{1}{2}$, 
but the automorphism group of the 
algebra is still finite.  In Section \ref{Monster algebras}, we show that subalgebras of unital Frobenius algebras are also unital.  This is used in Section \ref{naive} together with the na\"ive method to compute the axes and automorphism groups for the algebras for $S_4$.  Our nuanced method is developed in Section \ref{eigenvalue-0} and we use it to compute larger examples.  In Section \ref{twins and Jordan axes}, twin axes and Jordan axes are discussed.  Then, in Sections \ref{decomposition} and \ref{partial}, we give our much more advanced component method and discuss how to extend isomorphisms.

The remainder of the paper contains a sequence of examples of increasingly large dimension 
that we can handle using this approach. In Section \ref{larger}, we tackle two algebras 
for the group $S_5$ of dimensions $46$ and $61$. Unlike the naive and nuanced algorithms, 
the more general approach involves a combination of computer calculation and human input. 
In particular, the final argument identifying $\Aut(A)$ uses group-theoretic tools analysing 
the possible structure of $\Aut(A)$, based on the results of earlier calculations.  In Section \ref{further}, we complete two algebras for the group $\Aut(L_3(2))\cong PGL(2,7)$ of dimensions $49$ and $57$. Finally, in Section \ref{largest}, we do the two largest examples, the algebras of dimension $121$ and $151$ for $S_6$. 

To summarise, we demonstrate in this later part of the paper a much more powerful method, 
which we believe allows us to handle all or at least the majority of algebras that are 
currently known. There are some patterns in our calculations and arguments that suggest that 
some of the human input can be automated. However, even after that, one would need to identify 
$\Aut(A)$ from suitable group theoretic results.  

\medskip
{\bf Acknowledgement:} This work was supported by the Mathematical Center in Akademgorodok 
under the agreements No. 075-15-2019-1675 and 075-15-2022-281 with the Ministry of Science and 
Higher Education of the Russian Federation.

\section{Axial algebras} \label{background}

In this section we review the basics of axial algebras needed in the remainder of the paper.

A \emph{fusion law} is a pair $(\cF,\star)$, where $\cF$ is a set and $\star:\cF\times\cF\to 2^{\cF}$ is a binary operation on $\cF$ taking values in the power set of $\cF$. For simplicity, we will just speak of the fusion law $\cF$, assuming the operation $\star$.

It is common to represent fusion laws by tables similar to the group multiplication table, where in each cell corresponding to a pair $\lm,\mu\in\cF$, we simply list all elements of $\lm\star\mu$. The examples that will appear throughout the paper are shown in Figure \ref{fig:fusion laws}. They are the fusion laws $\cJ(\eta)$ and $\cM(\al, \bt)$ of \emph{Jordan} and \emph{Monster type}, respectively.

Let now $A$ be a commutative non-associative algebra over a field $\F$. Suppose that we are given a fusion law $\cF\subseteq\F$. An \emph{axis} for the fusion law $\cF$ is a non-zero idempotent $a\in A$ such that the adjoint map $\ad_a:A\to A$, sending each $u\in A$ to $au$, is semi-simple with all eigenvalues contained in $\cF$. Furthermore, it is required that
$$A_\lm(a)A_\mu(a)\subseteq A_{\lm\star\mu}(a),$$
for all $\lm,\mu\in\cF$. Here we denote by $A_\lm(a)$ the $\lm$-eigenspace of $\ad_a$, that is, 
$$A_\lm(a)=\{u\in A\mid au=\lm u\},$$
and, for $\cH\subseteq\cF$, we set
$$A_{\cH}(a)=\bigoplus_{\nu\in\cH}A_\nu(a).$$

Since $\ad_a(a)=aa=a$, we always assume that $1\in\cF$. The axis is said to be \emph{primitive} when $A_1(a)$ is $1$-dimensional, \ie it is spanned by $a$.

We say that $A$ is a (primitive) axial algebra for $\cF$ if $A$ is generated by a set $X$ of primitive axes. The set $X$ may or may not consist of all axes available in $A$. We note that for a primitive axis $a\in A$, we have that $A_1(a)A_\lm(a)=aA_\lm(a)=A_\lm(a)$ when $\lm\neq 0$ and $A_1(a)A_\lm(a)=0$ when $\lm=0$. Because of this, we will only consider fusion laws where $1\star\lm=\{\lm\}$ for $\lm\neq 0$ and $1\star 0=\emptyset$, as we see for example in Figure \ref{fig:fusion laws}.

We further note that, since our algebras are commutative, we only consider symmetric fusion laws, that is, we assume that $\lm\star\mu=\mu\star\lm$ for all $\lm,\mu\in\cF$.

Let us give some example of axial algebras which will appear later.  Let $(G, D)$ be a $3$-transposition group, that is $D$ is a conjugacy class of involutions in $G$ and $|cd| \leq 3$ for all $c,d \in D$.  Additionally, we require that $G = \la D \ra$.  Let $\F$ be a field of characteristic not $2$ and $\eta \in \F$, $\eta \neq 0,1$.  The \emph{Matsuo algebra} $M = M_\eta(G,D)$ is defined as follows: it has $D$ as a basis and the algebra multiplication is defined by
\[
c \cdot d := \begin{cases}
c & \mbox{if } d = c \\
0 & \mbox{if } |cd| = 2 \\
\frac{\eta}{2}(c+d-e) & \mbox{if } |cd|=3, \mbox{where } e := c^d = d^c
\end{cases}
\]
The elements of $D$ are primitive axes of Jordan type $\eta$ and so $M$ is an algebra of Jordan type $\eta$.  Another example of algebras of Jordan type are Jordan algebras, which arise only for $\eta = \frac{1}{2}$.

The Griess algebra for the Monster sporadic simple group is an algebra of Monster type $(\frac{1}{4}, \frac{1}{32})$.  We will meet other examples with the same fusion law later in the paper.

While the algebras we consider are non-associative, some elements in them may associate. Namely, 
we say that $u,v\in A$ \emph{associate} if for all $w\in A$ we have that $u(wv)=(uw)v$. Equivalently, $\ad_u$ and $\ad_v$ commute.

The fusion law $\cF$ is said to be \emph{Seress} if $0\in\cF$ and $0\star\lm\subseteq\{\lm\}$. (That is, $0$ and $1$ have essentially the same behaviour in $\cF$.)

\begin{lemma}[Seress]
If $A$ is an axial algebra for a Seress fusion law, then every axis $a\in A$ associates with $A_1(a)\oplus A_0(a)$.
\end{lemma}

For a proof, see \eg  \cite{KMS}.

Often an axial algebra admits a (non-zero) bilinear form that associates with the algebra product:
$$(uv,w)=(u,vw),$$
for all $u,v,w\in A$. Such a form is usually called a \emph{Frobenius form} and the algebra $A$ admitting it is said to be \emph{Frobenius} or \emph{metrisable}. The latter terminology is used especially in the case of $\F=\R$ or $\Q$, where the form may be positive definite and hence provide a metric. In any case, we will often refer to $(u,u)$ as the (square) \emph{length} of $u$. Typically, we expect the axes in $A$ to be non-singular, that is, to have a non-zero length, and whenever possible, we scale the form so that axes have length $1$.

The structure theory for axial algebras was developed in \cite{KMS}. The \emph{radical} of $A$ is the unique largest ideal of $A$ that does not contain any generating axes from $X$. When $A$ is Frobenius and axes are not singular, the radical of $A$ coincides with the radical of the Frobenius form, and so it is easy to find. We can further define a graph on $X$, where edges are the pairs $a,b\in X$ such that $(a,b)\neq 0$. We say that $A$ is \emph{connected} when this graph is connected. A connected axial algebra has all its proper ideals contained in the radical.

Axial algebras are inherently related to groups. Note that the Monster fusion law in Figure \ref{fig:fusion laws} is $C_2$-graded. Namely, if we split $\cF$ as follows: $\cF=\cF_+\cup\cF_-$, where $\cF_+=\{1,0,\al\}$ and $\cF_-=\{\bt\}$ then we can see that $\cF_+\star\cF_+,\cF_-\star\cF_-\subseteq\cF_+$ and $\cF_+\star\cF_-\subseteq\cF_-$. This grading results in the axial algebra $A$ also being graded by $C_2$, and therefore, for an axis $a\in X$, the linear map $\tau_a:A\to A$ acting as identity on $A_{\cF_+}(a)$ and as minus identity on $A_{\cF_-}(a)$ is an automorphism of $A$, called the\emph{ Miyamoto involution} (or sometimes, the \emph{tau involution}). The \emph{Miyamoto group} of $A$ is the group $\Miy(A)=\la\tau_a\mid a\in X\ra\leq\Aut(A)$.

Note that if $\phi\in\Aut(A)$ and $a$ is an axis then $a^\phi$ is also an axis. Hence we can extend our generating set of axes $X$ to $\bar X=\{a^\phi\mid a\in X,\phi\in\Miy(A)\}$. We say that $X$ is \emph{closed} when $\bar X=X$ and note that $\bar{\bar X}=\bar X$. Furthermore, we note closing the generating set of axes does not affect any of the main constructs, such as the radical and Miyamoto group.

The advantage of closing the set of axes is that $\Miy(A)$ acts on $\bar X$. The concept of a closed set $X$ of axes together with the corresponding Miyamoto group and tau map from $X$ to $\Miy(A)$ was generalised in \cite{MS} into a new class of objects called \emph{axets}. For our purposes it suffices to state that every closed set of axes constitutes an axet.

Finally, it was shown in \cite{FMS} (generalising earlier work in \cite{S,Axial1,IPSS}) that every axial algebra of Monster type $(\frac{1}{4},\frac{1}{32})$ generated by two axes is isomorphic to one of the eight \emph{Norton-Sakuma algebras}. The names of these algebras, $2A$, $2\B$, $3\A$, $3\C$, $4\A$, $4\B$, $5\A$, and $6\A$, reflect the way they arise within the Griess algebra. The exact structure constants for these algebras can be found, for example, in Table 3 in \cite{IPSS}. We will not quote it here in full in one place, and instead mention the relevant information where it is required.

To every axial algebra of Monster type $(\frac{1}{4},\frac{1}{32})$ we can associate its \emph{shape}, which is the record of $2$-generated subalgebras arising in it. While not ideal, this record is a useful way of identifying axial algebras. Note that in order to make the shape record shorter, only the subalgebras that are not uniquely identified by the axet are included in it. 

\section{Finiteness}\label{Finiteness}

The first question we need to discuss is whether the automorphism groups we want to study are finite or infinite. The answer will have major implications for the methods we employ.

\begin{theorem} \label{dimension zero}
Suppose that $\F$ is a field of characteristic other than $2$ and $A$ is a finite-dimensional axial algebra over $\F$ with fusion law $\cF\subseteq\F$. If $\frac{1}{2}\notin\cF$ then the automorphism group $\Aut(A)$ is finite.
\end{theorem}

\begin{proof}
Without loss of generality, we can assume that $\F$ is algebraically closed. Since $A$ is finite-dimensional, $G:=\Aut(A)$ is an affine algebraic group over $\F$. Hence, $G$ is finite if and only if it has dimension zero. 

Recall that the Lie algebra $L$ corresponding to $G$ acts on $A$ by derivations (see, for example, Corollary in Section 10.7 of \cite{H}). Let $d\in L$ and take an arbitrary axis $x\in A$. Then $d(x)=d(x^2)=d(x)x+xd(x)=2xd(x)$. Since $\F$ is not of characteristic $2$, we deduce that $\ad_x(d(x))=xd(x)=\frac{1}{2}d(x)$, that is, $d(x)$ is an eigenvector of $\ad_x$ corresponding to $\frac{1}{2}$. Since every eigenvalue of $\ad_x$ is contained in $\cF$ and $\frac{1}{2}\notin\cF$, it follows that $\frac{1}{2}$ is not an eigenvalue, and so we must have $d(x)=0$.  Hence every axis of $A$ is contained in $\ker(d)$.

It remains to notice that $B:=\ker(d)$ is a subalgebra of $A$. Indeed, it is clearly a subspace. Furthermore, if $u,v\in B$ then $d(uv)=d(u)v+ud(v)=0+0=0$, which shows that $B$ is closed for the algebra product and so it is a subalgebra. Since $A$ is generated by its axes and since all axes of $A$ are contained in $B$, we conclude that $A=B$, that is, $d=0$.  Therefore, the Lie algebra $L = 0$, and so $G$ is zero-dimensional and hence finite.
\end{proof}

This result is applicable to a wide range of known axial algebras. For example, 
we have the following corollaries of our theorem.

\begin{corollary} \label{Matsuo}
Suppose that $(G,D)$ is a finite $3$-transposition group and let $M=M_\eta(G,D)$ be the corresponding Matsuo algebra, where $\eta\notin\{0,\frac{1}{2},1\}$. Then $M$ has a finite automorphism group.
\end{corollary}

\begin{proof}
Indeed, according to \cite{Axial2}, since $\eta\notin\{0,1\}$, $M$ is a (finite-dimensional) algebra of Jordan type $\eta$, \ie the spectrum of $\ad_x$ for an axis $x\in M$ is contained in $\{1,0,\eta\}$. Since $\eta\neq\frac{1}{2}$, we conclude that $\frac{1}{2}\notin\{1,0,\eta\}$, and the claim follows from Theorem \ref{dimension zero}.
\end{proof}

Note that, since $\eta\neq\frac{1}{2}$, including in the set of axes of $M$ all primitive idempotents satisfying the fusion law of Jordan type $\eta$ leads to a potentially larger $3$-transposition group $(\hat G,\hat D)$. When $\hat D$ is indeed larger than $D$, the algebra $M$ must be isomorphic to the factor  $M_\eta(\hat G,\hat D)/I$ for a non-trivial ideal $I$ of $\hat M=M_\eta(\hat G,\hat D)$. In particular, the radical of $\hat M$ is non-zero, that is, $\eta$ must be critical for $\hat M$. (See \cite{KMS} for the concepts of radical and critical number). Such pairs $(G,D)$ and $(\hat G,\hat D)$ are called aligned, and they are currently being classified by Alharbi as part of his PhD project at the University of Birmingham. In particular, Alharbi found several series of aligned pairs, and hence examples of Matsuo algebras $M_\eta(G,D)$, whose automorphism groups exceed the automorphisms of $(G,D)$.

\medskip

Suppose again that $M=M_\eta(G,D)$ for a $3$-transposition group $(G,D)$. We will refer to the elements of $D$ as \emph{single axes} from $M$. We say that single axes $c$ and $d$ are \emph{orthogonal} if $cd=0$ in $M$. (This corresponds to non-collinearity in the Fischer space of $(G,D)$.) In \cite{j} (see also \cite{doubleMatsuo}), it was shown that the \emph{double axis} $x=c+d\in M$, for orthogonal $c$ and $d$, is an idempotent satisfying the fusion law of Monster type $(2\eta,\eta)$. Furthermore, for a flip (\ie automorphism of order $2$) $\sg$ of $(G,D)$, the \emph{flip subalgebra} $A(\sg)$ generated by all single and double axes contained in the fixed subalgebra $M_\sg=\{u\in M\mid u^\sg=u\}$ is a (primitive) algebra of Monster type $(2\eta,\eta)$.

\begin{corollary} \label{flip}
Suppose that $(G,D)$ is a \textup{(}finite\textup{)} $3$-transposition group and let $M=M_\eta(G,D)$ be the corresponding Matsuo algebra, where $\eta\notin\{0,\frac{1}{4},\frac{1}{2},1\}$. Then every subalgebra of $M$ generated by single and double axes has a finite automorphism group. In particular, for these values of $\eta$, all flip subalgebras have finite automorphism groups.
\end{corollary}

\begin{proof}
Indeed, since $\eta\neq\frac{1}{4},\frac{1}{2}$, we have that $\frac{1}{2}\notin\{1,0,2\eta,\eta\}$, and so the claim follows from Theorem \ref{dimension zero}, as both single and double axes obey the fusion law of Monster type $(2\eta,\eta)$.
\end{proof}

Note that this corollary has an additional exception compared to Corollary \ref{Matsuo}, namely, $\eta=\frac{1}{4}$. This leads to the following very interesting question.

\begin{question} \label{single and double}
Do there exist subalgebras of Matsuo algebras $M_{\frac{1}{4}}(G,D)$, generated by single and double axes, that have infinite automorphism groups?
\end{question}

In other words, is $\eta=\frac{1}{4}$ a true exception in Corollary \ref{flip} or can it be removed?

For the final corollary of Theorem \ref{dimension zero}, let us mention the class of algebras central to this paper: algebras of Monster type $(\al,\bt)$. Within this class, for $\al=\frac{1}{4}$ and $\bt=\frac{1}{32}$, we find the Griess algebra for the Monster simple group and, more generally, Majorana algebras of Ivanov \cite{I}. 

\begin{corollary} \label{our main case}
Every algebra of Monster type $(\al,\bt)$, with $\al\neq\frac{1}{2}\neq\bt$ has a finite automorphism group.
\end{corollary} 

In particular, the automorphism groups are finite for all finite-dimensional Majorana algebras, like say the Griess algebra, and in general for all the examples of axial algebras of Monster type $(\frac{1}{4},\frac{1}{32})$ that have been constructed.

This corollary is immediate from Theorem \ref{dimension zero}. In the second part of this paper, we endeavour to determine the exact automorphism groups for some of the known algebras in this class. As these automorphism groups are finite, we will develop computational tools using the computer algebra system MAGMA. 

In the remainder of this section, we develop an alternative method of showing finiteness of $\Aut(A)$, which could possibly be used when $\frac{1}{2}$ appears within a fusion law. As a trade-off, we need to impose rather severe restrictions on the fusion law $\cF$.

\begin{definition}
Let $B=\lla a,b\rra$ and $B'=\lla a',b'\rra$ be two $2$-generated algebras.  A \emph{pointed isomorphism} is an  isomorphism $\phi \colon B\to B'$ such that $\phi(a)=a'$ and $\phi(b)=b'$.
\end{definition}

Note that such an isomorphism is unique if it exists. Indeed, since the axes generate the 
algebra, the images of the axes identify the map.

Recall that a Frobenius form on an axial algebra is a non-zero bilinear form that associates with the algebra product:
$$(u,vw)=(uv,w),$$
for all $u,v,w\in A$. An element $u\in A$ is called singular if $(u,u)=0$.

Consider the following conditions on the fusion law $\cF$:
\begin{enumerate}
\item \label{cond1} up to pointed isomorphisms, there exist only finitely many $2$-generated 
primitive $\cF$-axial algebras $B=\lla a,b\rra$ admitting a Frobenius form such that 
$(a,a)=1=(b,b)$; and
\item \label{cond2} for all such $2$-generated algebras $B$, $\phi=(a,b)$ is not equal 
to $1$.
\end{enumerate}

\begin{theorem} \label{finite}
Suppose that the fusion law $\cF$ satisfies the above conditions \textup{\ref{cond1}} and \textup{\ref{cond2}}. Then every finite-dimensional primitive $\cF$-axial algebra $A$ admitting a Frobenius form, such that all generating axes are non-singular, has a finite automorphism group.
\end{theorem}

\begin{proof}
Let us first deal with the special case where all generating axes have length $1$ with respect to the Frobenius form. Let $X=X(A)$ be the axet from $A$ consisting of all primitive $\cF$-axes from $A$ having length $1$. It suffices to prove that $X$ is a finite set. Indeed, $X$ is clearly invariant under $\Aut(A)$ and, furthermore, $\Aut(A)$ acts faithfully on $X$, since $X$ generates $A$. Therefore, $\Aut(A)$ is isomorphic to a subgroup of the symmetric group of $X$, and so $\Aut(A)$ is finite when $X$ is.

Assume for a contradiction that $X$ is infinite. Consider the complete graph on $X$, where each edge $ab$ is ``coloured'' with the algebra $\lla a,b\rra$ viewed up to pointed isomorphisms. By \ref{cond1}, we have a finite number of colours in our infinite graph, so by the Infinite Ramsey's Theorem, there is an infinite subset $Y\subseteq X$ such that all subalgebras $\lla a,b\rra$, $a,b\in Y$, $a\neq b$ are isomorphic. In particular, $\phi=(a,b)$ is the same for all such $a$ and $b$. 

By \ref{cond2}, $\phi\neq 1$. Select $n$ such that $n>\dim(A)$ and $n\neq 1-\frac{1}{\phi}$ if $\phi\neq 0$. Let $Z\subset Y$ be of cardinality $n$ and consider the Gram matrix $T$ on the set $Z$. Clearly, $T=(1-\phi)I+\phi J$, where $I$ and $J$ are the identity and all-one matrices of size $n\times n$. It is easy to see that the eigenvalues of $J$ are $0$ (of multiplicity $n-1$) and $n$ (of multiplicity $1$). Correspondingly, $T$ has eigenvalues $1-\phi$ and $n\phi+(1-\phi)=(n-1)\phi+1$. By our choice, neither of these eigenvalues is zero, \ie $T$ is a non-degenerate matrix. However, this means that $Z$ spans a subspace of dimension $n$ in $A$, which is a contradiction since $\dim(A)<n$.

Now we turn to the general situation. Since $A$ is finite-dimensional, it is generated by a finite number of primitive axes. By our assumption, these axes have non-zero lengths $r_1,\ldots,r_t$. Let $X_i$ be the set of all primitive $\cF$-axes from $A$ of length $r_i$ and let $A_i=\lla X_i\rra$. Clearly, all $X_i$ and $A_i$ are invariant under $\Aut(A)$. Furthermore, $A_i$ satisfies the assumptions of the theorem and the Frobenius form on $A_i$ can be scaled so that all axes from $X_i$ have length $1$ in $A_i$. By the first part of our proof, $\Aut(A_i)$ is finite, that is, $\Aut(A)$ induces a finite group on $X_i$ and $A_i$. Consequently, $\Aut(A)$ also induces a finite group on $\bigcup_{i=1}^t X_i$. Since the latter contains a generating set of $A$, the action of $\Aut(A)$ on it is faithful, \ie $\Aut(A)$ is finite. 
\end{proof}

As you can see, the conditions \ref{cond1} and \ref{cond2} do not expressly prohibit $\frac{1}{2}\in\cF$. Hence we pose the following question.

\begin{question}
Are there fusion laws $\cF$ including $\frac{1}{2}$ that satisfy conditions \ref{cond1} and \ref{cond2} above?
\end{question}

In particular, we wonder whether Question \ref{single and double} can be answered using this 
technique or its extensions. Note that, for a given algebra $A$, the conditions \ref{cond1} 
and \ref{cond2} can be weakened by restricting them to only the $2$-generated subalgebras that 
actually do arise inside $A$.

In the following short section we provide an example of an algebra where $\frac{1}{2}$ is in the fusion law, but the automorphism group is still finite. This shows that Theorem \ref{dimension zero} is not the ultimate result, and its extensions may be useful.

\section{Example}\label{example}

In \cite{j}, a series $Q_k(\eta)$ of $k^2$-dimensional flip subalgebras was constructed inside $M_\eta(S_{2k},(1,2)^{S_{2k}})$. Let us focus on $\eta=\frac{1}{4}$ (not covered by Corollary \ref{flip}) and the smallest interesting case, $k=2$. In this case, $A:=Q_2(\frac{1}{4})$ is spanned by its basis of four axes: (a) two single axes, $s_1$ and $s_2$; and (b) two double axes, $d_1$ and $d_2$. The multiplication with respect to this basis is summarised in Table \ref{4-dim}, adapted from \cite{doubleMatsuo}. 
\begin{table}[ht]
\renewcommand{\arraystretch}{1.5}
\begin{center}
\begingroup
\setlength{\tabcolsep}{6pt}
\scalebox{.85}{
$\begin{tabu}[ht]{|c||c|c|c|c|}
\hline
&s_1&s_2&d_1&d_2\\
\hline\hline
s_1&s_1&0&\frac{1}{8}(2s_1+d_1-d_2)&\frac{1}{8}(2s_1+d_2-d_1)\\
\hline
s_2&0&s_2&\frac{1}{8}(2s_2+d_1-d_2)&\frac{1}{8}(2s_2+d_2-d_1)\\
\hline
d_1&\frac{1}{8}(2s_1+d_1-d_2)&\frac{1}{8}(2s_2+d_1-d_2)&d_1&
\frac{1}{4}(-s_1-s_2+d_1+d_2)\\
\hline
d_2&\frac{1}{8}(2s_1+d_2-d_1)&\frac{1}{8}(2s_2+d_2-d_1)&
\frac{1}{4}(-s_1-s_2+d_1+d_2)&d_2\\
\hline
\end{tabu}$}
\caption{The algebra $Q_2(\frac{1}{4})$}\label{4-dim}
\endgroup
\end{center}
\end{table}

Furthermore, $A$ inherits from its ambient Matsuo algebra $M_{\frac{1}{4}}(S_4,(1,2)^{S_4})$ a Frobenius form, whose Gram matrix with respect to the basis $\{s_1,s_2,d_1,d_2\}$ is
$$
\renewcommand{\arraystretch}{1.3}
\begin{pmatrix}
1&0&\frac{1}{4}&\frac{1}{4}\\
0&1&\frac{1}{4}&\frac{1}{4}\\
\frac{1}{4}&\frac{1}{4}&2&\frac{1}{2}\\
\frac{1}{4}&\frac{1}{4}&\frac{1}{2}&2
\end{pmatrix}.
$$

The determinant of this matrix, equal to $\frac{27}{8}$, is non-zero as long as 
the characteristic of $\F$ is not $3$. So we will additionally assume now that $\mathrm{char}(\F)\neq 3$. Then the radical of $A$ is zero, and so $A$ is simple. Finally, the algebra $A$ has an identity element $\one_A=\frac{2}{3}(s_1+s_2+d_1+d_2)$.

We note that the single axes $s_1$ and $s_2$ are of Jordan type $\frac{1}{4}$, while the double axes $d_1$ and $d_2$ are of Monster type $(\frac{1}{2},\frac{1}{4})$. While $\frac{1}{2}$ does not appear in the fusion law for $s_1$ and $s_2$, these two axes do not generate $A$ without $d_1$ and/or $d_2$, and so finiteness of $\Aut(A)$ does not follow from Theorem \ref{dimension zero}.
Yet finiteness holds here.

\begin{theorem}
The automorphism group $\Aut(A)$ of $A=Q_2(\frac{1}{4})$ is finite of order $4$.
\end{theorem}

\begin{proof}
Let $G=\Aut(A)$. We note that $G$ preserves the Frobenius form on $A$, since the latter is unique up to a scalar. It is easy to see that $G$ contains automorphisms switching the single axes $s_1$ and $s_2$ and, similarly, switching the double axes $d_1$ and $d_2$. So $|G|\geq 4$.

We claim that $s_1^G=\{s_1,s_2\}$. Note that all axes $s\in s_1^G$ are of Jordan type $\frac{1}{4}$ and satisfy the length condition $(s,s)=1$. Suppose by contradiction that $s$ is such an axis and $s\notin\{s_1,s_2\}$. Then Theorem (1.1) from \cite{Axial2} implies that the subalgebra generated by $s$ and $s_i$ is isomorphic either to $2\B$ or $3\C(\frac{1}{4})$ and so $(s,s_i)\in\{0,\frac{1}{8}\}$. If $s=\al s_1+\bt s_2+\gm d_1+\dl d_2$ then these conditions result in $(s,s_1)=\al+\frac{1}{4}(\gm+\dl)\in\{0,\frac{1}{8}\}$ and $(s,s_2)=\bt+\frac{1}{4}(\gm+\dl)\in\{0,\frac{1}{8}\}$. Additionally, $1=(s,s)=(\one_A,s^2)=(\one_A,s)=\frac{2}{3}(\frac{3}{2}\al+\frac{3}{2}\bt+3\gm+3\dl)=\al+\bt+2\gm+2\dl$. Expressing $\gm+\dl$ from this and substituting it in the first two relations, we get $\gm+\dl=\frac{1}{2}(1-\al-\bt)$ and $\al+\frac{1}{4}(\frac{1}{2}(1-\al-\bt))=\frac{1}{8}+\frac{7}{8}\al-\frac{1}{8}\bt\in\{0,\frac{1}{8}\}$, that is, $\frac{7}{8}\al-\frac{1}{8}\bt\in\{-\frac{1}{8},0\}$. Symmetrically, the second relation gives us that $-\frac{1}{8}\al+\frac{7}{8}\bt\in\{-\frac{1}{8},0\}$.

This results in four possible $1$-parameter families of solutions:
\begin{enumerate}[label={\upshape (\alph*)}]
\item $s=-\frac{1}{6}s_1-\frac{1}{6}s_2+(\frac{2}{3}-\dl)d_1+\dl d_2$;
\item $s=-\frac{7}{48}s_1-\frac{1}{48}s_2+(\frac{7}{12}-\dl)d_1+\dl d_2$;
\item $s=-\frac{1}{48}s_1-\frac{7}{48}s_2+(\frac{7}{12}-\dl)d_1+\dl d_2$;
\item $s=(\frac{1}{2}-\dl)d_1+\dl d_2$.
\end{enumerate}

Let us consider these possibilities. Since $s$ is an idempotent, we must have $s^2-s=0$. In case (a), $s^2-s=(\frac{1}{2}\dl^2-\frac{1}{3}\dl+\frac{5}{36})(s_1+s_2)+(\frac{1}{2}\dl^2+\frac{1}{6}\dl-\frac{5}{18})d_1+(\frac{1}{2}\dl^2-\frac{5}{6}\dl+\frac{1}{18})d_2$. (We skip the calculation details.) So $\dl$ has to be the root of all three quadratic polynomials arising as coefficients in this expression. The linear combination of the three polynomials with coefficients $2$, $-1$, and $-1$ equals to the constant polynomial $\frac{1}{2}$, which means that the polynomials have no common roots, and so there are no idempotents in case (a).

Similarly, in case (b), $s^2-s=(\frac{1}{2}\dl^2-\frac{7}{24}\dl+\frac{287}{2304})s_1+(\frac{1}{2}\dl^2-\frac{7}{24}\dl+\frac{35}{2304})s_2+(\frac{1}{2}\dl^2+\frac{5}{24}\dl-\frac{77}{288})d_1+(\frac{1}{2}\dl^2-\frac{19}{24}\dl+\frac{7}{288})d_2$. Taking the sum of the first two coefficients minus the sum of the last two coefficients, we obtain the constant $\frac{49}{128}$. So the four polynomials can only have a common root when the field $\F$ is of characteristic $7$. Furthermore, in characteristic $7$, the polynomials reduce to a simple form yielding $\dl=0$. However, substituting this into the expression for $s$ and reducing modulo $7$, we see that $s=s_1$, which is a contradiction. Symmetrically, we obtain a contradiction in case (c).

Finally, consider case (d). Here $s^2-s=(\frac{1}{2}\dl^2-\frac{1}{4}\dl)(s_1+s_2)+(\frac{1}{2}\dl^2+\frac{1}{4}\dl-\frac{1}{4})d_1+(\frac{1}{2}\dl^2-\frac{3}{4}\dl)d_2$. Again, take twice the first polynomial minus the sum of the last two polynomials to obtain the constant $\frac{1}{4}$. So there are no idempotents in this final case.

We have proved that $G=\Aut(A)$ leaves $\{s_1,s_2\}$ invariant. Now we can finish the proof. Let $H:=G_{s_1}$, the stabiliser of $s_1$ in $G$. Then $H$ fixes $s_1$ and $s_2$. Furthermore, it fixes $\one_A$, and so it fixes $s=\one_A-s_1-s_2=\frac{1}{3}(-s_1-s_2+2d_1+2d_2)$. Hence $H$ should also leave invariant the complement $\la s_1,s_2,s\ra^\perp=\la d_1-d_2\ra$. Thus, for $h\in H$, we have that $(d_1-d_2)^h=\lm(d_1-d_2)$ for some $\lm\in\F$. Also, we must have that $((d_1-d_2)^2)^h=\lm^2(d_1-d_2)^2$. On the other hand, $(d_1-d_2)^2=d_1+d_2-2\cdot \frac{1}{4}(-s_1-s_2+d_1+d_2)=\frac{1}{2}(s_1+s_2+d_1+d_2)\in\la s_1,s_2,s\ra$. Hence, $\lm^2=1$, which means that $\lm=\pm 1$. Consequently, $|H|\leq 2$, and the claim of theorem follows from here.
\end{proof}

The algebra $A$ is just one of the smaller examples of flip subalgebras of Monster type $(\frac{1}{2},\frac{1}{4})$, but this example suggests that the answer to Question \ref{single and double} may be negative. 

We note that the complete list of $2$-generated algebras of Monster type $(\frac{1}{2},\frac{1}{4})$ can be derived from \cite{fms}, and while the entire list is infinite, it might be possible to eliminate the infinite tail from consideration, if we can show that such $2$-generated subalgebras do not appear inside Matsuo algebras. Thus, we speculate that it may be possible to derive an answer to Question \ref{single and double} from a suitable extension of Theorem \ref{finite}.

\section{The algebras}\label{Monster algebras}

In the second part of the paper we develop methods for computing automorphism groups of axial algebras where they are known to be finite. We focus on the class of finite-dimensional algebras of Monster type $(\frac{1}{4},\frac{1}{32})$. This is because, on the one hand, by Corollary \ref{our main case}, these algebras always have finite automorphism groups and, on the other hand, there are currently around $200$ non-trivial examples of such algebras computed using the MAGMA realisation \cite{ParAxlAlg} of the expansion algorithm \cite{axialconstruction}. 

These algebras are normally realised over the field of rational numbers $\F=\Q$, but in our calculations we will allow any algebraic extension including the algebraic closure, if necessary, so that we can be sure that the automorphism groups cannot increase further for larger fields. We do not consider in this part of the paper fields of positive characteristics. 

In the definition of an axial algebra there is no requirement that the algebra be unital or admit a Frobenius form. However, a vast majority of them, at least among the known examples, are both unital and Frobenius (metrisable). In this brief section we discuss simple reasons why this is really not at all surprising.

First of all, a lot of the known algebras of Monster type $(\frac{1}{4},\frac{1}{32})$ are in fact subalgebras of the Griess algebra. Even though that is not how are defined and constructed, in retrospect, we can often find for small groups that their algebras do indeed arise within the Griess algebra. Note that the Griess algebra admits a positive definite Frobenius form (the Griess algebra is normally defined over $\R$).  Hence every subalgebra of the Griess algebra also admits a (positive-definite) Frobenius form.

The Griess algebra also is unital, \ie it has an identity element. We show below that this means that all of its subalgebras are also unital. Throughout the remainder of this section, $A$ is an axial algebra admitting a Frobenius form.

\medskip
We start with a useful lemma.

\begin{lemma} \label{module}
If $B$ is a subalgebra of $A$ then 
$$B^\perp=\{u\in A\mid(u,v)=0\mbox{ for all }v\in B\}$$
is a $B$-module, that is, $B^\perp B\subseteq B^\perp$.
\end{lemma}

\begin{proof}
Indeed, for $u\in B^\perp$ and $v,w\in B$, we have that $(uv,w)=(u,vw)=0$, since the form associates with the algebra product. Thus, $uv\in B^\perp$, as claimed.
\end{proof}

Let us now additionally assume that $A$ is unital, that is, it contains a multiplicative identity $\one$.

\begin{proposition} \label{unital}
If $B$ is a subalgebra of $A$ such that the Frobenius form is non-degenerate on $B$ then the orthogonal projection $\one_B$ of $\one$ onto $B$ is the multiplicative identity of $B$. In particular, $B$ is unital.
\end{proposition} 

\begin{proof}
Since $B$ is non-degenerate with respect to the Frobenius form, we have that $B^\perp\cap B=0$, and so $A=B^\perp\oplus B$ as a vector space. In particular, $\one=u+v$ for some $u\in B^\perp$ and $v\in B$, and clearly, $v$ is the projection of $\one$ onto $B$, that is, $\one_B=v$.

Note that, for all $w\in B$, we have that $w=\one w=uw+\one_Bw$. Furthermore, it is clear that $\one_Bw\in B$ and, by Lemma \ref{module}, $uw\in B^\perp$. Consequently, $uw$ and $\one_Bw$ are the projections of $w$, respectively, to $B^\perp$ and $B$. Since $w\in B$, we conclude that $uw=0$ and $w\one_B=w$. This shows that $\one_B$ is indeed the identity of $B$ and $u=\one-\one_B$ is in the annihilator of $B$.
\end{proof}

Let us also note the interesting fact that the annihilator of $B$ is non-trivial if $B$ does not contain $\one$.

\begin{corollary}
All subalgebras of the Griess algebra are metrisable and unital.
\end{corollary}

\begin{proof}
Since the form on the Griess algebra is positive definite, it is non-zero on every subalgebra $B$, and clearly, the restriction of the form is a (positive-definite) Frobenius form on $B$. Also, the Griess algebra is unital, and so $B$ is unital by Proposition \ref{unital}.
\end{proof}

As we already mentioned, the majority of known algebras of Monster type $(\frac{1}{4},\frac{1}{32})$ are subalgebras of the Griess algebra, so this explains their unitality and metrisability. For the algebras that are not subalgebras of the Griess algebra, unitality and positive-definiteness of the Frobenius form can be checked directly. (See the tables in \cite{axialconstruction}.)

To summarise, all the algebras, whose automorphism groups we compute in this paper, are unital and admit a (positive-definite) Frobenius form. Furthermore, all their $\Q$-subalgebras, whether axial or not, are unital too. 

We now turn to the constructive computation of the automorphism group.

\section{Naive approach} \label{naive}

Let us suppose that it is known that $\Aut(A)$ is finite for an axial algebra $A$. How can $\Aut(A)$ be found explicitly? 

As the axes generate $A$, the automorphism group acts faithfully on the set of all axes, including the known ones and, possibly, new, unknown ones. Hence finding $\Aut(A)$ is essentially the same as finding all axes in $A$. On the surface, looking for axes is easy. The most simply-minded approach would be to select a basis $u_1,\ldots,u_n$ in $A$, write an arbitrary vector as a linear combination $u=\sum_{i=1}^n x^iu_i$, where $x^1,\ldots,x^n$ are independent indeterminates\footnote{Not to be confused with the powers of a single indeterminate $x$.}, and obtain a system of quadratic equations from the equality $u^2=u$. If $\gm_{ij}^k$ are the structure constants of $A$ with respect to the basis $u_1,\ldots,u_n$ (\ie we have that $u_iu_j=\sum_{k=1}^n\gm_{ij}^ku_k$ for all $i$ and $j$) then these quadratic equations are 
$$\sum_{i,j=1}^n\gm_{ij}^kx^ix^j-x^k=0,$$
for $k=1,2,\ldots,n$. Solving this system of quadratic equations, we find all idempotents in $A$ and then we can check them individually to find all primitive axes for the target fusion law $\cF$.

Of course, in reality this plan is workable only for very small algebras $A$. Even when the variety of idempotents is of dimensional zero, by B\'ezout's Theorem, the number of solutions of this system, counted with multiplicities, equals $2^n$, \ie it is exponential in the dimension of $A$. With the available solvers (we use MAGMA \cite{MAGMA}), only algebras of dimension of up to about $10$ can be handled using this simply-minded approach. In Table \ref{naive_results} below\footnote{The $\infty$ in the table means that the computation completes but takes an unreasonable amount of time.}, we list the algebras done in this naive way.
\begin{table}[!h] 
\renewcommand{\arraystretch}{2}
\centering 
\begin{tabular}{cccccc}\hline\hline
$G$&axet&$\mathrm{dim}$&shape&$\mathrm{Aut}(A)$&time (s)\\
\hline\hline
$S_4$ & 3+6 & 9 & 3C2A & $S_4$ & $\infty$\\ 
 & 6 & 6 & 3C2 & $S_4$ & 0.52\\ 

\hline\hline
\end{tabular}\caption{Naive algorithm}\label{naive_results}
\end{table}
Note that both these cases have already been completed by Castillo-Ramirez \cite{AssocSubs}, as part of the classification of associative subalgebras of these algebras.

We can substantially improve on this simply-minded approach by requiring that the idempotents have the prescribed length $r$. (In case of algebras of Monster type $(\frac{1}{4},\frac{1}{32})$ this must be $r=1$.) Adding the requirement that the length of idempotent be $r$ gives an additional quadratic equation, which significantly reduces the set of solutions. 

However, since our algebras $A$ are metrisable and unital, we can do even better.

\begin{lemma}
If $A$ unital and metrisable then $(u,u)=(\one,u)$ for every idempotent $u\in A$.
\end{lemma}

\begin{proof}
Indeed, $(u,u)=(\one u,u)=(\one,uu)=(\one,u)$.
\end{proof}

(Note that we already used this observation in the proof of Theorem \ref{4-dim}.) This lemma means that, instead of a quadratic equation $(u,u)=r$, we can use a linear equation $(\one,u)=r$. Though it looks to reduce the dimension of the problem only by one, surprisingly, adding this linear relation substantially speeds up the calculation of all idempotents and thus significantly increases the algebra dimensions that can be handled using this approach.

In Table \ref{less naive}, we recorded some results obtained using the naive algorithm with 
length restriction (with $r=1$). For comparison, we also included the two algebras from Table 
\ref{naive_results}, which now complete significantly faster.

\begin{table}[!h]
\renewcommand{\arraystretch}{2} 
	\centering 
	\begin{tabular}{ccccccr}\hline\hline
		$G$&axet&$\mathrm{dim}$&shape&new&$\mathrm{Aut}(A)$&time (s)\\\hline\hline
		&$6$&13&3A2B&&$S_4$&9.96\\
		&6&6&3C2B&&$S_4$&0.01\\
		$S_4$&$3+6$&12&4A3C2B&&$S_4$&2.11\\
		&$3+6$&13&4B3A2A& &$S_4$&10.86\\
		&$3+6$&16&4B3A2B&3&$S_4$&1473.09\\
		&$3+6$&9&4B3C2A& &$S_4$&0.05\\
		&$3+6$&12&4B3C2B&3&$S_4$&1.72\\\hline 
		$S_5$&10&10&3C2B&&$S_5$&0.18\\ \hline 
		$S_6$&15&15&3C2B&&$S_6$&342.75\\ 
		\hline\hline		
	\end{tabular}
\caption{Naive algorithm with length}\label{less naive}
\end{table}

To summarise, for small algebras $A$ only, it is possible to find the entire variety of idempotents from $A$ (or idempotents of a given length $r$). Note that these methods can also be used to find all idempotents within a small subspace of a larger $A$.

\section{The $\mathbf{0}$-eigenvalue subalgebra}\label{eigenvalue-0}

In this section we discuss a method which can be used to enumerate all axes in a less naive way. While we will rely on some properties specific to the fusion law $\cM(\frac{1}{4},\frac{1}{32})$, these properties are hopefully general enough to allow the application of the same ideas to a wider class of axial algebras.

Let $A$ be an algebra of Monster type $(\frac{1}{4},\frac{1}{32})$ and let $X=X(A)$ be its axet, \ie the set of known axes invariant under the action of the known group of automorphisms, $G$. Typically, $G$ is the group that $A$ was constructed from. The issue is whether $X$ is the full set of axes from $A$ and, equivalently, whether $G$ is the full automorphism group of $A$. 

Select $a\in X$. Recall that $\cM(\frac{1}{4},\frac{1}{32})$ is a Seress fusion law, and in particular, as $0\star 0=\{0\}$, we have that $U=A_0(a)$ is a subalgebra of $A$.

As above, we would like to decide whether $X$ contains all primitive axes of Monster type $(\frac{1}{4},\frac{1}{32})$ from $A$, or whether there are further such axes $b$.  According to \cite{FMS}, $B=\lla a,b\rra$ is a Norton-Sakuma algebra, and it is unital.  Let $\one_B$ be the multiplicative identity of $B$. 

\begin{lemma}
The element $z=\one_B-a$ is an idempotent contained in $U$.
\end{lemma}

\begin{proof}
Indeed, $z^2=(\one_B-a)^2=\one_B^2-2\one_B a+a^2=\one_B-a=z$, \ie $z$ is an idempotent. Moreover, $za=(\one_B-a)a=\one_B a-a^2=a-a=0$, and so $z\in A_0(a)=U$.
\end{proof}

This suggests the following possible two-step approach to our problem: 
\begin{enumerate}[label=(\arabic*)]
\item \label{step1} we first find all possible idempotents $z\in U$; and then
\item \label{step2} we find $B$ (and $b$) inside $A_1(a+z)=A_1(\one_B)$.
\end{enumerate}

Since $U=A_0(a)$ is smaller that $A$ (typically, $\dim(U)$ is between a third and a half of $\dim(A)$), step \ref{step1} can be done, say, as in Section \ref{naive}. For step \ref{step2}, we can, pragmatically, expect $A_1(a+z)$ to be equal or at least close to $B$ in all cases, where $a+z$ is $\one_B$ for some $B$. It is also likely that $A_1(a+z)$ is small for all remaining idempotents $z$, as annihilators (and $A_1(a+z)$ is simply the annihilator of $a+z$) tend to be small. So in all cases, it is likely that the naive method of finding the idempotents $b$ can be used at step \ref{step2} as well.

As we explained in Section \ref{naive}, an important improvement comes from the known length of $z$. The lengths of $\one_B$ for all Norton-Sakuma algebras $B$ are summarised in Table \ref{length of one}, where the values in the second row are taken from \cite{NortSakIdemps}.
\begin{table}[ht]
\centering
\renewcommand{\arraystretch}{2}
\begin{tabular}{ccccccccc}\hline\hline 
Type of $B$&$2\A$&$2\B$&$3\A$&$3\C$&$4\A$&$4\B$&$5\A$&$6\A$\\\hline\hline 
$(\one_B,\one_B)$&$\frac{12}{5}$&$2$&$\frac{116}{35}$&$\frac{32}{11}$&$4$&$\frac{19}{5}$&$\frac{32}{7}$&$\frac{51}{10}$\\\hline 
$(z,z)$&$\frac{7}{5}$&$1$&$\frac{81}{35}$&$\frac{21}{11}$&3&$\frac{14}{5}$&$\frac{25}{7}$&$\frac{41}{10}$\\\hline 
\end{tabular}
\caption{Identity length of the Norton-Sakuma algebras} \label{length of one}
\end{table}

We note the following.

\begin{lemma}
We have that $(z,z)=(\one_B,\one_B)-(a,a)=(\one_B,\one_B)-1$.
\end{lemma}

\begin{proof}
Indeed, $(z,z)=(\one_B-a,\one_B-a)=(\one_B,\one_B)-2(\one_B,a)+(a,a)=(\one_B,\one_B)-2(a,a)+(a,a)=
(\one_B,\one_B)-(a,a)=(\one_B,\one_B)-1$, since $(\one_B,a)=(\one_B,a^2)=(\one_B 
a,a)=(a,a)$. 
\end{proof}

This explains the values in the third row of Table \ref{length of one}.

\medskip
In view of this, we only need idempotents $z\in U$, whose length is in the bottom row of Table \ref{length of one}. So we can can use the naive method with length restriction.

This method works really well with the algebras of dimension up to about $35$. The algebras we completed using this method are listed in Table \ref{NuancedResults}. Again, for comparison, we included two long cases from Table \ref{less naive}, for the group $G=S_4$ of dimension $16$ and for $G=S_6$ of dimension $15$. As you see, now they are quite quick.
\begin{table}[h!]
\renewcommand{\arraystretch}{2}
\centering
\caption{Results obtained from the nuanced algorithm}\label{NuancedResults}
\bigskip
\begin{tabular}{ccccccr}
\hline\hline 
$G$&axet&$\mathrm{dim}$&shape&new&$\mathrm{Aut}(A)$&time (s)\\
\hline\hline
$S_4$&$3+6$&23&4A3A2A&7&$C_2\times S_4$&4.07\\
&$3+6$&25&4A3A2B& &$S_4$&53.05\\
&$3+6$&16&4B3A2B&3&$S_4$&0.78\\
\hline 
$A_5$&15&26&3A2A& &$S_5$&1.53\\
&15&20&3C2A& &$S_5$&0.29\\
&15&21&3C2B& &$S_5$&1.21\\
\hline
$S_5$&10&10&3C2B&&$S_5$&0.07\\
&$10+15$&36&4B& &$S_5$&14912.95\\
\hline
$L_3(2)$&21&21&4B3C& &$\Aut(L_3(2))$&1.36\\
\hline
$S_6$&15&15&3C2B&&$S_6$&1.18\\
\hline 
$S_3\times S_3$&$3+3+9$&18&3A2A& &$S_3\wr C_2$&0.84\\
&$3+3+9$&25&3A2B&7&$C_2\times (S_3\wr C_2)$&87.20\\
\hline
$(S_4\times S_3)^{+}$&$3+18$&24&3C3C3C2A& &$C_3:S_3: S_4$&78.87\\
&$3+18$&27&3C3C3C2B&3&$C_3:S_3: S_4$&$\infty$\\ 
\hline 
$S_7$&21&21&3C2B& &$S_7$&2371.58\\ 
\hline\hline
\end{tabular}
\end{table}

In some rare instances, the dimension of the variety of idempotents in $U$ was positive (\ie we had infinitely many idempotents) even after we added the length relation. Then an additional relation was needed and it was developed as follows. 

Depending on the type of $B$, we know which eigenvalues $a$ has on $B$. Equivalently, we know the eigenvalues $\lm\in\{1,0,\frac{1}{4},\frac{1}{32}\}$ such that $B\cap A_\lm(a)\neq 0$. 

\begin{lemma}
If $a$ has eigenvalue $\lm$ on $B$ then $z$ has eigenvalue $1-\lm$ in its action on $A_\lm(a)$. Consequently, $\det(\ad_z|_W-(1-\lm)\Id)=0$, where $W=A_\lm(a)$, $\ad_z|_W$ is the restriction of the adjoint $\ad_z$ to $W$ and $\Id$ is the identity map on $W$.
\end{lemma}

\begin{proof}
First, we note that $0\star\lm\subseteq\{\lm\}$ as $\cM(\frac{1}{4},\frac{1}{32})$ is Seress. This means that $zA_\lm(a)\subseteq A_\lm(a)$, since $z\in U=A_0(\lm)$. Thus, $\ad_z$ leaves $W=A_\lm(a)$ invariant, and so we can indeed consider the restriction of $\ad_z$ to $W$.

Assuming that $a$ has eigenvalue $\lm$ on $B$, consider $0\neq w\in B\cap A_\lm(a)=B\cap W$. Then $zw=(\one_B-a)w=\one_B w-aw=w-\lm w=(1-\lm)w$.  Thus, as $w\in W$, the restriction of $\ad_z$ to $W$ indeed has eigenvalue $1-\lm$.
\end{proof}

The new relation $\det(\ad_z|_W-(1-\lm)\Id)=0$ is a polynomial relation of degree at most $\dim(W)$ in the indeterminates $x^1,\ldots,x^n$. Note that the relation is trivial for $\lm=1$, so it only makes sense to use $\lm\in\{0,\frac{1}{4},\frac{1}{32}\}$. When $\dim(W)$ is large, this additional relation is not good, as it significantly increases the time and memory demand. So it should not be used routinely, but only where there is no other way. In reality, these extra relations were only needed in the case where $B\cong 4\A$ (\ie $(z,z)=3$) and then, in all cases, the determinant relation coming from $\lm=\frac{1}{32}$ reduces the variety dimension to zero.

\section{Twins and Jordan axes} \label{twins and Jordan axes}

In this section we discuss some observations we can draw from the completed cases. 

\subsection{Twins}

First of all, we have cases where we find additional axes. However, in all of these cases the full automorphism group is either equal or only slightly bigger than the initial group $G$. What happens is that we get additional axes which produce the same Miyamoto involutions.

\begin{definition}\label{twin_defn}
	Let $A$ be an $\mathcal{F}$-axial algebra where $\mathcal{F}$ is a $C_2$-graded fusion law, and let $a,b$ be distinct axes. Then we say that $a$ and $b$ are \textit{twins} if $\tau_a=\tau_b$. 
\end{definition} 

Clearly, being twins is an equivalence relation, and so, at least in principle, we can also have three or more axes leading to the same Miyamoto involution. The next lemma gives some properties of twin axes in the case where $\mathcal{F}$ is the fusion law of Monster type $\mathcal{M}(\alpha,\beta)$. 

For a subspace $W\subseteq A$, we denote the annihilator of $W$ by $\Ann(W)$; that is,
$$\Ann(W)=\{u\in A\mid uw=0\mbox{ for all }w\in W\}.$$

\begin{lemma}\label{twin_props}
	Let $A$ be an axial algebra of Monster type $\mathcal{M}(\al,\bt)$. Then the following are equivalent:
	\begin{enumerate}[label={\upshape (\alph*)}]
	\item axes $a,b$, $a\neq b$, are twins;  
	\item the eigenspaces $A_\bt(a)$ and $A_\bt(b)$ are equal.
    \end{enumerate}
    Furthermore, these two imply
    \begin{enumerate}[resume]
	\item The element $b-a$ is in the annihilator of $A_\bt(a)$; equivalently, $b=a+u$, for some $u\in\Ann(A_\bt(a))$. 
	\end{enumerate} 
\end{lemma}

\begin{proof}
For an axis $a\in A$, write $A:=A_+\oplus A_-$, where $A_+:=A_1(a)\oplus A_0(a)\oplus A_\al(a)$ and $A_-:=A_\bt(a)$. Then every $u\in A$ can be written as $u=u_++u_-$ for unique $u_+\in A_+$ and $u_-\in A_-$. Note that $u^{\tau_a}=u_+-u_-$ and so $[u,\tau_a]:=u^{\tau_a}-u=(u_+-u_-)-(u_++u_-)=-2u_-\in A_-$. Since $\F$ is not of characteristic $2$, we conclude that $A_\bt(a)=A_-=[A,\tau_a]$ is the commutator space. Now if $b\in A$ is a second axis then $a$ and $b$ are twins if and only if $A_\bt(a)=[A,\tau_a]=[A,\tau_b]=A_\bt(b)$. Thus, (a) and (b) are equivalent.

Next, if $W:=A_\bt(a)=A_\bt(b)$ and $w\in W$ then $(a-b)w=aw-bw=\bt w-\bt w=0$. Thus, $a-b$ annihilates this $w$ and hence also all of $W$. Hence, (b) (and so also (a)) implies (c). 
\end{proof}

We have already mentioned that annihilators tend to be small, especially for substantial subspaces $W$, such as $W=A_\bt(a)$. They only require basic linear algebra for their computation, as the conditions are linear. Hence property (c) gives us an efficient method to determine all axes $b$ that are twins with $a$. Namely, all such $b$ are contained in the coset $a+\Ann(A_\bt(a))$ within the low-dimension subspace $\la a,\Ann(A_\bt(a))\ra$. Hence we can use the naive method to find all such twin axes $b$. This works really well in the algebras that we tried; in fact, in all algebras we tried the dimension of the annihilator $\Ann(A_\bt(a))$ was no more than $10$, and in a large majority it was below $4$. So the calculation was quite quick.

Now recall that the set $X$ of known axes (the axet) is invariant under the action of the known group of automorphisms, $G$. Hence $X$ decomposes as a union of orbits of $G$.

\begin{lemma}
If axes $a$ and $b$ of $A$ are twins then so are $a^g$ and $b^g$ for all $g\in\Aut(A)$.
\end{lemma}

\begin{proof}
Indeed, note that $(\tau_a)^g=\tau_{a^g}$ for an axis $a\in A$ and an automorphism $g\in\Aut(A)$. Consequently, if $a$ and $b$ are twins then $\tau_{a^g}=(\tau_a)^g=(\tau_b)^g=\tau_{b^g}$. Thus, $a^g$ and $b^g$ are also twins.
\end{proof}

In particular, this is true for all $g\in G$, and so we only need to compute twins for one axis from each orbit of $G$ on $X$. For the remaining axes twins can be found using the action by $G$.

The conclusion of this discussion is that we can find all twins from the start, prior to employing more complicated and time-consuming techniques.

Before we switch to the next topic, let us point out that the annihilator technique can be used to check whether a given involution $\tau\in\Aut(A)$ is the Miyamoto involution for an axis. Indeed, suppose that there is a (possibly unknown) axis $x$ such that $\tau_x=\tau$. If the fusion law for $x$ is such that $\tau_x$ negates a single eigenspace $A_\lm(x)$ (as is the case for the fusion law of Monster type), we see that $x-\lm \one_A$ must annihilate $[A,\tau]$, that is, $x\in\la\Ann([A,\tau]),\one_A\ra$. Thus, if the annihilator $\Ann([A,\tau])$ is small, so is the subspace $\la\Ann([A,\tau]),\one_A\ra$, and so we can use the brute force method to find all idempotents in this subspace and then check if any one of them is an axis satisfying the required fusion law.

This is also a very efficient method and it allows us to find axes for known involutions in $G$, when they exist.

\subsection{Jordan axes}

We also encountered a small number of cases, where we found additional axes for which the $\frac{1}{32}$-eigenspace was trivial. When the $\bt$-eigenspace for an axis $a$ of Monster type $(\al,\bt)$ is trivial, the Miyamoto involution $\tau_a$ is trivial. However, in this case we have another involution, $\sg_a$ (called the \emph{sigma involution}) that negates $A_{\al}(a)$. Indeed, this is because such axes obey the tighter fusion law of Jordan type $\al$. In our case $\al=\frac{1}{4}$, and so the sigma involution negates the $\frac{1}{4}$-eigenspace. We will call such axes \emph{Jordan axes}.

We found an effective method of finding all Jordan axes in an algebra $A$. It is based on the following lemma, specific for the case $(\al,\bt)=(\frac{1}{4},\frac{1}{32})$, where we know all $2$-generated algebras.

\begin{proposition}
If $A$ is an algebra of Monster type $(\frac{1}{4},\frac{1}{32})$ then all Jordan axes are contained in the fixed subalgebra $A_G$, where $G=\Miy(A)$ is the Miyamoto group of $A$.
\end{proposition}

\begin{proof}
It was shown in \cite{FMS}, that every $2$-generated algebra of Monster type $(\frac{1}{4},\frac{1}{32})$ is isomorphic to one of the Norton-Sakuma algebras. All these algebras are symmetric, and in particular, $|a^D|=|b^D|$, where $a$ and $b$ are two axes and $D=\la\tau_a,\tau_b\ra$. If $a$ is a Jordan axis then $\tau_a=1$, and so $|b^D|=1$. It follows that $|a^D|=1$, that is, the Jordan axis $a$ is fixed by $\tau_b$ for all axes $b$. Thus, $a$ is fixed by the entire Miyamoto group $G$.
\end{proof}

We call the fixed subalgebra $A_G$ the \emph{$\sg$-subalgebra} and denote it $\sg(A)$. In the examples that we tried, $\sg(A)$ is always quite small. The maximum dimension we found was just $7$, so the Jordan axes can be found easily, and this can be done in the very beginning, just like with the twins.

Let us also record the following related fact.

\begin{proposition}
If $a$ is a Jordan axis then $\sg_a$ centralises the Miyamoto group $G=\Miy(A)$. 
\end{proposition}

\begin{proof}
We have already seen that $a\in\sg(A)=A_G$. Therefore, $\sg_a^g=\sg_{a^g}=\sg_a$ for all $g\in G$, that is, $\sg_a$ centralises $G$.
\end{proof}

As we can expect, the Miyamoto group $G$ often constitutes the bulk of the full automorphism group $\Aut(A)$. Hence, typically, the elements $\sg_a$ for Jordan axes $a$ end up in the centre of $\Aut(A)$.

Also, Jordan axes, when they exist in $A$, lead to twins for other axes. We note that when $\sg_a\neq 1$, the fixed subalgebra $A_{\sg_a}\neq A$ and so, as $A$ is generated by axes, there will be many axes that are not fixed by $\sg_a$.

\begin{proposition}
If $a\in A$ is a Jordan axis with $\sg_a\neq 1$ then $b$ and $b^{\sg_a}$ are twins for all axes $b$ not contained in the fixed subalgebra $A_{\sg_a}$.
\end{proposition}

\begin{proof}
Indeed, $\tau_{b^{\sg_a}}=\tau_b^{\sg_a}=\tau_b$, since $\sg_a$ and $\tau_b\in G=\Miy(A)$ commute.
\end{proof}

So it is no surprise that in all the cases in our tables in which the algebra contains a Jordan axis, it also contains twins. The following question seems interesting.

\begin{question}
Are there algebras of Monster type containing more that one Jordan axis?
\end{question}

So far we have not seen such examples, although it is plausible that they may exist. We note that, when $\al\neq\frac{1}{2}$, the Jordan axes generate a Matsuo subalgebra (or a factor of a Matsuo algebra) inside $A$. Furthermore, the $3$-transposition group generated by the sigma involutions would act on $A$ by permuting classes of twins.

Can it be that $\sg_a=1$ for a Jordan element $a$? This would mean that both $\al$- and $\bt$-eigenspaces of $\ad_a$ are trivial and such an axis would obey the associative fusion law involving just $\{1,0\}$. In particular, the Seress Lemma implies that all such axes $a$ are in the centre\footnote{Recall that the \emph{centre} of a non-associative algebra $A$ consists of all elements that commute and associate with all elements of $A$. Furthermore, the commutativity condition is moot for axial $A$, because they are commutative.} of $A$. 

Let us now see more in detail how a Jordan element behaves with respect to all other axes from $A$ when $(\al,\bt)=(\frac{1}{4},\frac{1}{32})$.

\begin{proposition}
Suppose that $A$ is of Monster type $(\frac{1}{4},\frac{1}{32})$ and $a\in A$ is a Jordan axis. Then, for an axis $b\in A$, $b\neq a$, we have that $\lla a,b\rra\cong 2\B$ if $b\in A_{\sg_a}$, and $\lla a,b\rra\cong 2\A$ if $b\notin A_{\sg_a}$.
\end{proposition}

\begin{proof}
Indeed, $2\A$ and $2\B$ are the only Norton-Sakuma algebras where $\tau_a$ fixes $b$ and, symmetrically, $\tau_b$ fixes $a$. Furthermore, $\sg_a$ fixes $b$ in $2\B$ and it does not fix $b$ in $2\A$.
\end{proof}

In particular, $a$ annihilates all axes without twins. This gives potentially additional strong conditions for finding Jordan elements if we meet significantly bigger $\sg(A)$ in future examples.

\section{Decompositions} \label{decomposition}

As we already mentioned, the method involving the $0$-eigenspace subalgebra works well for algebras up to dimension $36$. In particular, it was sufficient for our initial goal: to find the automorphism groups of the algebras for $S_4$. However, our ambition grew as well, and so we want and need to do even larger algebras. Luckily, the $0$-eigenspace method allows a natural generalisation.

Suppose that $A$ is an $\cF$-axial algebra for a $C_2$-graded fusion law $\cF$, and suppose we have a set of axes $Y:=\{a_1,a_2,\ldots,a_k\}\subseteq A$ such that each $\tau_i:=\tau_{a_i}$ fixes all axes $a_j \in Y$. Then the group $E:=\la\tau_1,\tau_2,\ldots,\tau_k\ra$ is an elementary abelian $2$-group. For a tuple $(\lm_1,\lm_2,\ldots,\lm_k)\in\cF^k$, define
$$A_{(\lm_1,\lm_2,\ldots,\lm_k)}(Y):=\bigcap_{i=1}^k A_{\lm_i}(a_i).$$
Let $K=G_{a_1,a_2,\ldots,a_k}$ be the joint stabiliser of the axes from $Y$ in the known group $G$ acting on $A$. Then $K$ centralises each $\tau_i$ and so it centralises $E$. Also, $E\leq K$.

We note the following.

\begin{proposition} \label{modules}
Each subspace $A_{(\lm_1,\lm_2,\ldots,\lm_k)}(Y)$ is invariant under $K$. Furthermore, if $\cF$ is 
Seress then
\begin{enumerate}
\item $U:=A_{(0,0,\ldots,0)}(Y)$ is a subalgebra; and 
\item every $W:=A_{(\lm_1,\lm_2,\ldots,\lm_k)}(Y)$ is a $U$-module; that is, $UW\subseteq W$.
\end{enumerate} 
\end{proposition}

\begin{proof}
Since $K$ fixes every $a_j$, it leaves $A_{\lm_j}(a_j)$ invariant. Hence also $W=A_{(\lm_1,\lm_2,\ldots,\lm_k)}(Y)=\bigcap_{j=1}^k A_{\lm_j}(a_j)$ is invariant under $K$, proving the first claim.

Furthermore, in a Seress fusion law, we have that $0\star 0=\{0\}$, which means that every $A_0(a_i)$ is a subalgebra. Hence $U=\bigcap_{i=1}^kA_0(a_i)$ is also a subalgebra, since the intersection of subalgebras is a subalgebra.  

As $\cF$ is Seress,  $A_0(a_j)A_{\lm_j}(a_j)\subseteq A_{0\star\lm_j}(a_j)\subseteq A_{\lm_j}(a_j)$. Thus, $UW\subseteq A_0(a_j)A_{\lm_j}(a_j)\subseteq A_{\lm_j}(a_j)$ for all $j$. Hence $UW\subseteq\bigcap_{j=1}^kA_{\lm_j}(a_j)=W$.
\end{proof}

It is easy to see that the subspaces $A_{(\lm_1,\lm_2,\ldots,\lm_k)}(Y)$ form a direct sum decomposition within the algebra $A$. In one important case, this is a decomposition of the entire algebra $A$. 

\begin{proposition}
Suppose that $\cF$ is Seress and $\lla a_i,a_j\rra\cong 2\B$ for all $i\neq j$. \textup{(}Equivalently, $a_ia_j=0$.\textup{)} Then
$$A=\bigoplus_{(\lm_1,\lm_2,\ldots,\lm_k)\in\cF^k}A_{(\lm_1,\lm_2,\ldots,\lm_k)}(Y).$$
\end{proposition}

\begin{proof}
We prove this by induction on $k=|Y|$. If $k=1$ then this is true since axes are semi-simple. Suppose $k\geq 2$ and the claim is true whenever the number of axes is less than $k$. Let $Y'=\{a_1,a_2,\ldots,a_{k-1}\}$. By the inductive hypothesis, we have that
$$A=\bigoplus_{(\lm_1,\lm_2,\ldots,\lm_{k-1})\in\cF^{k-1}}A_{(\lm_1,\lm_2,\ldots,\lm_{k-1})}(Y').$$
As $a_ia_k=0$ for all $i=1,\ldots,k-1$, we have that $a_k\in A_{(0,0,\ldots,0)}(Y')$. By Proposition~\ref{modules}~(b), we have that $a_kA_{(\lm_1,\lm_2,\ldots,\lm_{k-1})}(Y')\subseteq A_{(\lm_1,\lm_2,\ldots,\lm_{k-1})}(Y')$. That is, $\ad_{a_k}$ acts on $A_{(\lm_1,\lm_2,\ldots,\lm_{k-1})}(Y')$. 

Now recall that $\ad_{a_k}$ is semi-simple on $A$, which means that its minimal polynomial is multiplicity-free. So the minimal polynomial of $\ad_{a_k}$ acting on $W':=A_{(\lm_1,\lm_2,\ldots,\lm_{k-1})}(Y')$ is also multiplicity-free, \ie $\ad_{a_k}$ is also semi-simple on $W'$. It follows that 
$$W'=\bigoplus_{\lm_k\in\cF}(W'\cap A_{\lm_k}(a_k))=\bigoplus_{\lm_k\in\cF}
A_{(\lm_1,\lm_2,\ldots,\lm_k)}(Y).$$
Clearly, this implies the claim for $k$ axes, and so the proposition follows by induction.
\end{proof}

This suggests the following approach, where we continue to assume that $\cF$ is Seress and axes in $Y$ are pairwise orthogonal. Suppose that, instead of proving that $G=\Aut(A)$, we just show that $K=G_{a_1,a_2,\ldots,a_k}$ coincides with the full joint stabiliser $\hat K:=\Aut(A)_{a_1,a_2,\ldots,a_k}$ of the axes $a_i\in Y$. The hope is that this statement identifies a significant subgroup of $\Aut(A)$ and this allows us then to deduce that $\Aut(A)=G$ by group-theoretic methods.

This could be organised as follows. First, we study the much smaller subalgebra $U=A_{(0,0,\ldots,0)}(Y)$ and find its full automorphism group. We note that $U$ may contain some axes, but in general it does not have to be an axial algebra, just like in Section \ref{eigenvalue-0}, $U=A_0(a)$ does not have to be axial.

Secondly, we try to see which automorphisms of $U$ extend to the other pieces $W$ of the direct sum decomposition of $A$. In particular, when we take the identity automorphism of $U$, this allows us to decide whether $\hat K:=\Aut(A)_{a_1,a_2,\ldots,a_k}$ acts faithfully on $U$.

So now we have to develop methods for extending automorphisms of $U$ to direct summands $W$. Recall that by Proposition \ref{modules}, each $W$ is a module for $U$. The following observation does not require a proof.

\begin{proposition} \label{extension}
Let $\phi$ be an automorphism of $U=A_{(0,0,\ldots,0)}(Y)$ and $W=A_{(\lm_1,\lm_2,\ldots,\lm_k)}(Y)$. If $\hat\phi$ is an extension of $\phi$ to the entire $A$, fixing all $a_i\in Y$, then $\psi:=\hat\phi|_W$ satisfies
$$u^\phi w^\psi=(uw)^\psi$$
for all $u\in U$ and $w\in W$.
\end{proposition}

We note that $\psi$ is a linear transformation of $W$ and the condition in this proposition is linear in both $u$ and $w$, so it can be checked for bases of $U$ and $W$.

Furthermore, setting $l:=\dim(U)$ and $m=\dim(W)$, we can treat the $m^2$ entries of the matrix of $\psi$ as indeterminates. Then the above conditions turn into a system of $lm^2$ linear equations, which can be solved to give us possible extensions of $\phi$ to this particular summand $W$.

%\bigskip
%{\color{red} Should we insert here a discussion of the trick with the Kronecker product?}
%\bigskip

Let us now see how this works for concrete algebras $A$.

\section{Larger examples} \label{larger}

\subsection{The $\mathbf{46}$-dimensional algebra for $\mathbf{A_5}$ of shape $\mathbf{3A2B}$}

In this case, the known axet in $A$ consists of $15$ axes and the known automorphism group of $A$ is $G\cong S_5$ containing $G_0=\Miy(A)\cong A_5$ as a subgroup of index two. 

\begin{computation}\label{prelim A5 3A2A}
Using the methods from Section $\ref{twins and Jordan axes}$, we verify the following:
\begin{enumerate}
\item $A$ contains no Jordan axes;
\item $A$ contains no twins, so the $\tau$ map is a bijection between the $15$ known axes and $15$ involutions in $G_0$;
\item there are no axes in $A$ corresponding to the $10$ involutions in $G\setminus G_0$.
\end{enumerate}
\end{computation}

Select a Sylow $2$-subgroup $E\cong 2^2$ of $G_0$. Let $\tau_1$, $\tau_2$ and $\tau_3$ be the involutions from $E$. By (b) above, there exist unique axes $a_1$, $a_2$, $a_3$, such that $\tau_i=\tau_{a_i}$ for all $i$. Let $Y=\{a_1,a_2,a_3\}$. Note that $K=C_G(E)=E$. We aim to prove that also $\hat K=E$, where $\hat K=\Aut(A)_{a_1,a_2,a_3}$ is the joint stabiliser of $a_1$, $a_2$, and $a_3$ in $\Aut(A)$.

We let $N=N_G(E)$ be the set-wise stabiliser of $Y$ in $G$. Note that $N$ induces $N/E\cong S_3$ on the set $Y$.

\begin{computation}
The decomposition of $A$ corresponding to $Y$ is as follows:
\begin{enumerate}
\item $U:=A_{(0,0,0)}(Y)$ is of dimension $7$;
\item the remaining non-zero summands $W=A_{(\lm_1,\lm_2,\lm_3)}(Y)$ are:
\begin{enumerate}
\item $A_{(1,0,0)}(Y)=\la a_1\ra$, $A_{(0,1,0)}(Y)=\la a_2\ra$, and $A_{(0,0,1)}(Y)=\la a_3\ra$;
\item $A_{\left( \frac{1}{4}, 0, \frac{1}{4} \right)}(Y)$, $A_{\left( \frac{1}{4}, \frac{1}{4}, 0 \right)}(Y)$, and $A_{\left( 0, \frac{1}{4}, \frac{1}{4} \right)}(Y)$, each of dimension $1$;
\item $A_{\left( \frac{1}{4}, 0, 0 \right)}(Y)$, $A_{\left( 0, \frac{1}{4}, 0\right)}(Y)$, and $A_{\left( 0, 0, \frac{1}{4} \right)}(Y)$, each of dimension $2$;
\item $A_{\left(\frac{1}{4}, \frac{1}{32}, \frac{1}{32}\right)}(Y)$, $A_{\left(\frac{1}{32}, \frac{1}{32}, \frac{1}{4}\right)}(Y)$, and  $A_{\left(\frac{1}{32}, \frac{1}{4}, \frac{1}{32}\right)}(Y)$ of dimension $3$;
\item $A_{\left(\frac{1}{32}, \frac{1}{32}, 0\right)}(Y)$, $A_{\left(0, \frac{1}{32}, \frac{1}{32}\right)}(Y)$, and $A_{\left(\frac{1}{32}, 0, \frac{1}{32}\right)}(Y)$ of dimension $6$.
\end{enumerate}
\end{enumerate}
These summands are listed in triples, as these are the orbits under the action of $N$. 
\end{computation}

Next we need to find $\Aut(U)$ and we employ the method from Section \ref{naive} for this. 

\begin{computation}
We have that
\begin{enumerate}
\item $U$ contains exactly three idempotents of length $2$, $v_1$, $v_2$, and $v_3$;
\item each $v_i$ is a primitive axis of Monster type $(\frac{4}{11},\frac{1}{11})$ in $U$;
\item $V:=\lla v_1,v_2,v_3\rra$ is of dimension $3$ isomorphic to $3\C(\frac{4}{11})$.
\end{enumerate}
\end{computation}

Since the $v_i$ are of Monster type in $U$ (they have a more complicated fusion law in the whole of $A$), each $v_i$ induces a Miyamoto involution on $U$, and this gives us a group $H:=S_3$ acting on $V$ and $U$. Clearly, this means that also $\Aut(U)$ induces $S_3$ on $V$, but it can, in principle, induce a larger group on $U$, since $U>V$. So we need further calculations, namely, we need to see whether $\Aut(U)$ contains non-identity elements fixing all three idempotents $v_i$. 

\begin{computation}\label{autU A5 3A2A}
We further compute that:
\begin{enumerate}
\item the identity $\one_U$ of $U$ has the (square) length $11$ (\ie $(\one_U,\one_U)=11$);
\item the identity $\one_V$ of $V$ has length $\frac{11}{2}$;
\item  $U$ contains exactly eight idempotents of length $\frac{11}{2}$: $\one_V$, $\one_U-\one_V$, and six further idempotents $u_i$; 
\item each $u=u_i$ is uniquely identified by the triple of values $(u,v_j)$, $j=1,2,3$; for three of them these values include $2$ and two $1$s, and for the other three, they include $0$ and two $1$s; and
\item the eight idempotents of length $\frac{11}{2}$ generate $U$.
\end{enumerate}
\end{computation}

Suppose $\phi\in\Aut(U)$ fixes $v_1$, $v_2$, and $v_3$. Clearly, $\phi$ fixes $\one_V$ and $\one_U-\one_V$. Furthermore, in view of Computation \ref{autU A5 3A2A} (d), $\phi$ fixes all $u_i$, as it preserves the Frobenius form on $U$. Now, by Computation \ref{autU A5 3A2A} (e), $\phi=1$, and so $\Aut(U)=H\cong S_3$.

The next step is to see whether the elements of $H$ can be extended to the entire $A$. We note that we are only interested in the extensions fixing the axes $a_i\in Y$, and so each such automorphism should leave every summand of our decomposition invariant. 

We will focus on the $3$-dimensional summands $W_1=A_{\left(\frac{1}{4},\frac{1}{32},\frac{1}{32}\right)}(Y)$, $W_2=A_{\left(\frac{1}{32},\frac{1}{4},\frac{1}{32}\right)}(Y)$, and $W_3=A_{\left(\frac{1}{32},\frac{1}{32},\frac{1}{4}\right)}(Y)$. This is because of the following fact we verified computationally.

\begin{computation}\label{extensions A5 3A2A}
The following hold:
\begin{enumerate}
\item $\lla W_1,W_2,W_3\rra=A$;
\item the identity automorphism of $U$ admits a $1$-dimensional space of extensions to each $W_i$;
\item for randomly chosen elements $w_i\in W_i$, $i=1$, $2$, and $3$, and $u\in U$, we have
\begin{enumerate}
\item $(w_i^2,u)\neq 0$; and 
\item $(w_1w_2,w_3)\neq 0$.
\end{enumerate}
\end{enumerate}
\end{computation}

In part (b), we utilised the method from Proposition \ref{extension}. This calculation yields the following.

\begin{lemma} \label{A5kernel}
The joint stabiliser $\hat K$ of $a_1$, $a_2$ and $a_3$ acts on $U$ with kernel $E$.
\end{lemma}

\begin{proof}
Let $\phi\in\hat K$ be acting as identity on $U$. First of all, by Computation \ref{extensions A5 3A2A} (a), if $\phi$ is identity on the union of $W_1$, $W_2$, and $W_3$ then $\phi=1$. Also, by (b), $\phi$ acts as a scalar, say $\mu_i$, on each $W_i$.

Next we use the facts in (c). We note that $(w_i^2,u)=((w_i^\phi)^2,u^\phi)=((\mu_iw_i)^2,u)=\mu_i^2(w_i^2,u)$. Since $(w_i^2,u)\neq 0$ by (c)(i), we conclude that $\mu_i^2=1$, that is, $\mu_i=\pm 1$. Furthermore, $(w_1w_2,w_3)=(w_1^\phi w_2^\phi,w_3^\phi)=(\mu_1 w_1\mu_2 w_2,\mu_3 w_3)=\mu_1\mu_2\mu_3(w_1w_2,w_3)$. Since $(w_1w_2,w_3)\neq 0$ by (c)(ii), we have that $\mu_1\mu_2\mu_3=1$. 

Thus, only the following triples of values $(\mu_1,\mu_2,\mu_3)$ are possible: $$\{(1,1,1),(1,-1,-1),(-1,1,-1),(-1,-1,1)\}.$$ 
Manifestly, the first of these is realised by the identity automorphism and the latter three are realised by the Miyamoto involutions $\tau_1$, $\tau_2$, and $\tau_3$, respectively. Thus, the kernel of $\hat K$ acting on $U$ coincides with $E$. 
\end{proof}

It remains to see that none of the non-identity automorphisms of $U$ extend to $A$ while fixing $a_1$, $a_2$, and $a_3$. This requires additional computational checks. Recall that $N\cong S_4$ is the normaliser of $E$ in $G$. Since $N$ stabilises $Y=\{a_1,a_2,a_3\}$ as a set, it leaves $U=A_{\left(0,0,0\right)}(Y)$ invariant, while permuting the other components of our decomposition.

\begin{computation}\label {induced_xn N A5 3A2A}
~
\begin{enumerate}
\item $N$ induces on $U$ the full group $H=\Aut(U)\cong S_3$;
\item an involution from $H$ does not have any extensions to two of the components $W_i$;
\item an element of order $3$ from $H$ does not have extensions to either of the components $W_i$.
\end{enumerate}
\end{computation}

Now we can identify $\hat K$.

\begin{lemma}
We have that $\hat K=K=E$.
\end{lemma}

\begin{proof}
By Computation \ref{induced_xn N A5 3A2A} (b), the involution $\tau\in H$ we tried does not have a required extension to two of the components $W_i$, which means that $\hat K$ contains no elements inducing $\tau$ on $U$. Since, by (a), all involutions from $H$ are conjugate under the action of $N$, none of them is induced by $\hat K$.

Similarly, by (c), an element of order $3$ cannot be extended to an element of $\hat K$. Thus, $\hat K$ induces on $U$ the trivial group; \ie $\hat K$ is fully in the kernel when acting on $U$, and so, by Lemma \ref{A5kernel}, $\hat K=E$.
\end{proof}

The above computation-based argument allowed us to identify the full joint stabiliser $\hat K=\Aut(A)_{a_1,a_2,a_3}$. Now we use finite group theory arguments to deduce that $\Aut(A)=G\cong S_5$.

First of all, by Corollary \ref{our main case}, $\hat G=\Aut(A)$ is a finite group. We first show that the soluble radical of $\hat G$ is trivial, \ie $\hat G$ contains no abelian normal subgroups. By contradiction, suppose that $Q$ is an abelian minimal normal subgroup of $\hat G$. Then it is an elementary abelian $p$-subgroup for some prime $p$. It can be viewed as a vector space over $\F_p$ and a $\hat G$-module.

Recall the notation $G_0\cong A_5\geq E$ for the subgroup of $G$ of index $2$. Let $R$ be a minimal non-trivial subgroup of $Q$ invariant under $G_0$. Then $R$ is irreducible as a $G_0$-module. 

\begin{lemma} \label{no centraliser}
We have that $C_R(E)=1$.
\end{lemma}

\begin{proof} 
Indeed, $E=C_{\hat G}(E)\geq C_R(E)$. If the latter is non-trivial, we have that $G_0\cap Q\neq 1$, which is clearly a contradiction, since $G_0\cap Q\unlhd G_0$ and $G_0\cong A_5$ is a non-abelian simple group. 
\end{proof}

In particular, this implies the following.

\begin{corollary}
We have that $p\neq 2$.
\end{corollary}

\begin{proof}
Indeed, if $p=2$ then $C_R(E)$ cannot be trivial, as both $R$ and $E$ are $2$-groups.
\end{proof}

So $p$ is an odd prime. In fact, we can say a lot more than that. Recall that Miyamoto 
involutions in the automorphism groups of algebras of Monster type 
$(\frac{1}{4},\frac{1}{32})$ form a class of $6$-transpositions. (See \eg  Corollary 2.10 from 
\cite{3-gen}.)

\begin{lemma} \label{6-trans}
We have that $p\in\{3,5\}$.
\end{lemma}

\begin{proof}
Take $1\neq e\in E$. We note that $e=\tau_i=\tau_{a_i}$ for some $i$, so $e$ is a Miyamoto involution. It follows from Lemma \ref{no centraliser} that $e$ cannot act trivially on $R$ (as all non-identity elements of $E$ are conjugate in $G$). In particular, since $p\neq 2$, there must be an element $1\neq r\in R$ inverted by $e$. Then $|ee^r|=|er^{-1}er|=|(r^{-1})^er|=|r^2|=p$, and since $e$ belongs to a class of $6$-transpositions, we must have that $p\leq 6$.
\end{proof}

We now consider separately the cases of $p=3$ and $p=5$ and achieve a contradiction in both of them using the known modular character tables of $A_5$. First, we need the following observation which applies in both cases.

\begin{lemma} \label{mult of three}
Let $n=\dim(R)$. Then $n=3k$ for some $k\in\N$. Furthermore, the eigenvalues of each $1\neq e\in E$ on $R$ are $1$ (with multiplicity $k$) and $-1$ (with multiplicity $2k$).
\end{lemma}

\begin{proof}
Clearly, since $e$ has order $2$ and $p$ is odd, we have that $e$ has eigenvalues $1$ and $-1$ on $R$. So we just need to determine multiplicities.

Since $E$ is abelian, $R$ admits a basis $\{r_1,r_2,\ldots,r_n\}$ with respect to which all $e\in E$ are diagonal. By Lemma \ref{no centraliser}, since the product of the three involutions from $E$ is one, we must have for each $i$ that two involutions $e\in E$ invert $r_i$ and one involution centralises $r_i$. Thus, the total number of $-1$s in the matrices of the three involutions is $2n$ and the total number of $1$s is $n$. Since the three involutions are conjugate in $G_0$, they have the same eigenvalue multiplicities, and so the multiplicity $k$ of the eigenvalue $1$ for each $e$ is $\frac{n}{3}$ and the multiplicity of $-1$ for each $e$ is $\frac{2n}{3}=2k$.
\end{proof}

\begin{lemma} \label{not 3}
We have that $p\neq 3$.
\end{lemma}

\begin{proof}
In characteristic $3$, the group $A_5$ has irreducible modules in dimension $1$, $3$, $3$, and $4$. By the preceding lemma, since $R$ is irreducible, $n=\dim(R)=3$ and $k=1$. Let $d\in G_0$ be an element of order $5$ inverted by $e\in E$. Note that $3^3-1=27-1=26$ is not a multiple of $5$. Hence $d$ fixes a non-zero vector $r\in R$. It follows that $R=C_R(d)\oplus[R,d]$, where $\dim(C_R(d))=1$ and $\dim([R,d])=2$. Clearly, $e$ leaves both $C_R(d)$ and $[R,d]$ invariant. Since $e$ inverts $d$, it cannot act on $[R,d]$ as a scalar, so $e$ has both eigenvalues $1$ and $-1$ on $[R,d]$. This implies, by the above lemma, that $e$ must have eigenvalue $-1$ on $C_R(d)$. However, this means that $e$ inverts $t:=dr$, which is an element of order $15$, and we obtain that $|ee^t|=|(t^{-1})^et|=|t^2|=15$, which is a contradiction, since $e$ is a $6$-transposition. 
\end{proof}

Similarly, we also rule out the second case.

\begin{lemma}
We have that $p\neq 5$.
\end{lemma}

\begin{proof}
In characteristic $5$, $A_5$ has irreducible modules of dimension $1$, $3$, and $5$. Hence 
again $n=\dim(R)=3$ and $k=1$. Let now $d\in G_0$ be an element of order $3$ that is inverted 
by $e\in E$. Again $R=C_R(d)\oplus[R,d]$ with $\dim(C_R(d))=1$ and $\dim([R,d])=2$. We 
similarly deduce that $e$ has eigenvalues $1$ and $-1$ within $[R,d]$, and so it must invert 
$1\neq r\in C_R(d)$. Thus, $e$ inverts $t=dr$ of order $15$, which again, as in the proof of 
Lemma \ref{not 3}, contradicts the fact that $e$ is a $6$-transposition.
\end{proof}

We have achieved our goal. Since we obtained a contradiction in all cases, our assumption that $Q$ was abelian cannot hold. 

\begin{corollary}
Every minimal normal subgroup $Q$ of $\hat G$ is non-abelian.
\end{corollary}

Since $Q$ is a minimal normal subgroup, we must have that $Q=L\times L\times\cdots\times L$ for a non-abelian simple group $L$. In fact, we will shortly see that $Q=L$ is simple.

Let $S\cong D_8$ be a Sylow $2$-subgroup of $G$ containing $E$.

\begin{lemma}
We have that $S$ is a Sylow $2$-subgroup of $\hat G$.
\end{lemma}

\begin{proof}
Let $\hat S$ be a Sylow $2$-subgroup of $\hat G$ containing $S$. If $S<\hat S$ then also $S<N_{\hat S}(S)$. Let $t\in N_{\hat S}(S)$. By the Computation \ref{prelim A5 3A2A} (c), there are no axes corresponding to the involutions in $S\setminus E$. In particular, the involutions from $E$ cannot be conjugate to the involutions from $S\setminus E$. It follows that $t$ must normalise $E$, that is, $t\in N_{\hat G}(E)=\Aut(A)_{a_1,a_2,a_3}=E$. This is a contradiction and it shows that $S=\hat S$. 
\end{proof}

It is well known that no non-abelian simple group can have a Sylow $2$-subgroup of order less that $4$. Since the $2$-part of $|\hat G|$ is $2^3$ by the lemma we have just proved, we conclude that the following must be true.

\begin{corollary}
We have that $Q=L$ is a simple group. Furthermore, $G_0\leq Q$.
\end{corollary}

\begin{proof}
The first claim is now clear. Since $|S|=2^3$, $|E|=2^2$, and the order of $Q\cap S$ is at least $2^2$, we must have that $Q\cap E\neq 1$. However, so $G_0\cap Q\neq 1$, and so the claim follows, since $G_0\cong A_5$ is simple.
\end{proof}

We can now prove the ultimate result.

\begin{proposition} \label{S5}
The automorphism group of $A$ is $H\cong S_5$.
\end{proposition}

\begin{proof}
It follows from the lemmas above that $\hat G=\Aut(A)$ contains a unique minimal normal subgroup $Q=L$, which is a non-abelian simple group. Recall the notation $S\cong D_8$ for a Sylow $2$-subgroup of $G\cong S_5$ containing $E$. By Computation \ref{prelim A5 3A2A} (c), the involutions in $S\setminus E$ do not correspond to any axes, and in particular, they are not conjugate in $\hat G$ to the elements of $E$. By Thompson's Transfer Theorem (see \eg Theorem 12.1.1 in \cite{KS}), $\hat G$ has an index $2$ subgroup $\hat G_0$, containing $E$, but not $S$. Clearly, $Q$, being simple, is contained in $\hat G_0$. It follows that $E$ is the Sylow $2$-subgroup of $Q$, that is, $Q$ is a simple group with an elementary abelian Sylow $2$-subgroup of order $4$. It follows from \cite{W} that $Q\cong L_2(q)$ for $q\equiv 3,5\mbox{ mod }8$. It remains to bound the value of $q$. Note that an involution from $L_2(q)$, for odd prime power $q$, inverts tori of size $\frac{q-1}{2}$ and $\frac{q+1}{2}$. Since the involutions from $E$ are $6$-transpositions, it follows that $\frac{q+1}{2}\leq 6$, that is, $q\leq 11$. Since $q\equiv 3,5\mbox{ mod }8$, we have that $q=5$, or $11$. Thus, we just need to rule out the case $q=11$. 

Note that all involutions in $L_2(11)$ are conjugate. If $Q\cong L_2(11)$ then all involutions in $Q$ are tau involutions. Furthermore, since we do not have twins for the initial axes, we cannot have them for any axes. Therefore, the axet for $\hat G$ must be isomorphic to the axet of all involutions from $Q\cong L_2(11)$. However, the latter axet has pairs of involutions generating a subgroup $D_{12}$, which means that the corresponding pairs of axes generate the Norton-Sakuma algebra $6\A$. Within this subalgebra, we see that pairs of axes corresponding to commuting involutions (there is a single orbit of such pairs in $L_2(11)$) generate $2\A$. However, in our algebra $A$, the axes corresponding to two involutions of $E$ are orthogonal, that is, they generate $2\B$. This is a contradiction, since we have a single orbit on such pairs.

Thus, $q\neq 11$, and so $q=5$, and this means that $Q\cong L_2(5)\cong A_5$. We conclude that $Q=G_0$. Finally, since $Q$ is the only minimal normal subgroup of $\hat G$, we must have that $Q=F^*(\hat G)$ (the generalised Fitting subgroup), and this means that $C_{\hat G}(Q)=Z(Q)=1$. Thus, $\hat G$ is isomorphic to a subgroup of $\Aut(Q)\cong\Aut(A_5)\cong S_5$. This finally yields the desired result that $\hat G=G\cong S_5$.   
\end{proof}

\subsection{The $\mathbf{61}$-dimensional algebra for $\mathbf{S_5}$ of shape $\mathbf{4A}$}

This algebra $A$ is initially constructed from an axet with $10+15$ axes (two orbits) with the Miyamoto group $G_0:=Miy(A)\cong S_5$. 

\begin{computation}
~
\begin{enumerate}\label{S5 4A prelim}
\item $A$ contains a unique Jordan axis $d$; the corresponding sigma involution $\sg_d$ fixes the axes in the orbit $15$ and produces twins for the orbit $10$;
\item $A$ contains no further twins for the known axes; 
\item in the group $G=\la\sg_d\ra\times G_0\cong 2\times S_5$, no further further involution corresponds to an axis of $A$.
\end{enumerate}
\end{computation}

Hence we start with the group $G\cong 2\times S_5$ defined above and axet with $1+10+10+15$ axes. We aim to prove that $G=\Aut(A)$ and, correspondingly, the $36$ axes that we know are all the axes that $A$ contains.

From our discussion, the orbit $15$ on axes is bijectively mapped by the tau map onto the $15$ involutions in the subgroup $A_5$ of $G_0$, while each of the twin orbits $10$ is mapped bijectively onto the second class of involutions from $G_0$.

We apply the decomposition method with $Y=\{d\}$. 

\begin{computation}
~
\begin{enumerate}
\setcounter{enumi}{3}
\item $U=A_{(0)}(Y)=A_0(d)$ is of dimension $46$;
\item the only further summands in the decomposition are $A_{(1)}(Y)=A_1(d)=\la d\ra$ and $A_{\left(\frac{1}{4}\right)}(Y)=A_{\frac{1}{4}}(d)$ of dimension $14$.
\end{enumerate}
\end{computation}

It is easy to see that $U=A_{(0)}(Y)=A_0(d)$ is the algebra of type $3\A2\B$ generated by our orbit $15$, and it is exactly the algebra whose full automorphism group we found in the preceding subsection. Namely, $\Aut(U)\cong S_5$, and in particular, $G$ induces on $U$ its full automorphism group. Note that $G$ fixes $d$ and so it acts on $U$ and each summand $W$ of the decomposition of $A$. 

Let $\hat G=\Aut(A)$. Clearly, $\hat G$ fixes $d$, as it is the only axis of Jordan type in $A$. 
Hence $\hat G$ acts trivially on $\la d\ra$ and it also acts on the remaining two parts of the decomposition, $U=A_1(d)$ and $A_{\frac{1}{4}}(d)$. We have already stated that $G$ induces on $U$ its full automorphism group. Hence, in order to show that $\hat G=G$, we need to establish that $\hat G$ and $G$ have the same kernel acting on $U$. For $G$, this kernel coincides with the group $\la\sg_d\ra$ of order $2$. 

\begin{computation}\label{extension S5 4A}
~
\begin{enumerate}
\item The identity automorphism of $U$ has a $1$-dimensional space of extensions to $W=A_{\frac{1}{4}}(d)$; namely, they act on $W$ by scalars; and
\item a random element of $W$ does not square to zero. 
\end{enumerate}
\end{computation} 

For Computation \ref{extension S5 4A} (a), we used the method from Proposition \ref{extension} (see also the discussion after the proposition). Part (b) is just a direct calculation. 

Let $\phi\in\hat G$ be acting trivially on $U$. Then, by (a), $\phi$ multiplies every element of $W$ by a certain scalar $\lm$. Let $w\in W$ be the element from (b). Then we know that $w^2\neq 0$ and, by the fusion law for the Jordan axis $d$, we have that $w^2\in A_1(d)\oplus A_1(d)=\la d\ra\oplus U$. Hence $\phi$ fixes $w^2$. On the other hand, $(w^2)^\phi=(w^\phi)^2=(\lm w)^2=\lm^2 w^2$. Since $w^2\neq 0$, it follows that $\lm^2=1$, that is, $\lm=\pm 1$.

The value $\lm=1$ corresponds to $\phi$ equal to the identity automorphism of $A$, while $\lm=-1$ corresponds to $\phi=\sg_d$. Thus, there are exactly two extensions of the identity automorphism of $U$ to the entire $A$. Both of these extensions are elements of $G$. Therefore, we conclude that $\Aut(A)=\hat G=G$.

We have now obtained the main result of this subsection. 

\begin{proposition}
The full automorphism group of the $61$-dimensional algebra $A$ of shape $4\A$ coincides with $G\cong 2\times S_5$. Furthermore, $A$ contains exactly $1+10+10+15$ axes.
\end{proposition}

\section{Partial decomposition} \label{partial}

Towards the end of Section \ref{decomposition} and Section \ref{larger}, we focussed on the case where the selected axes were pairwise orthogonal, \ie they annihilated each other. In this case we obtained a decomposition of the entire algebra as a direct sum of joint eigenspaces. When we allow non-orthogonal axes, \ie when $\lla a_i,a_j\rra\not\cong 2\B$ for at least one pair of selected axes, the decomposition into joint eigenspaces is only partial. Let $$A^\circ=\bigoplus_{(\lm_1,\lm_2,\ldots,\lm_k)\in\cF^k}A_{(\lm_1,\lm_2,\ldots,\lm_k)}(Y).$$

We can additionally introduce $A^\sharp:=(A^\circ)^\perp$. When the Frobenius form is non-degenerate on $A^\circ$ (for example, when it is positive definite on $A$), we will have that $A=A^\circ\oplus A^\sharp$, that is, this new piece completes the decomposition of $A$. We note that it is also a module for $U=A_{(0,0,\ldots,0)}(Y)$.

\begin{proposition}
The subspace $W=A^\sharp$ satisfies $UW\subset W$; that is $W$ is a $U$-module.
\end{proposition}

\begin{proof}
Indeed, for $w\in W$, $u\in U$ and $v\in A^\circ$, we have $(uw,v)=(wu,v)=(w,uv)=0$, since $uv\in A^\circ$, as each component of $A^\circ$ is a module for $U$. Hence $uw\in(A^\circ)^\perp=A^\sharp=W$.
\end{proof}

Hence we can treat $W=A^\sharp$ just like any other component of our decomposition.

\section{Further examples} \label{further}

\subsection{The $\mathbf{49}$-dimensional algebra for $\mathbf{L_3(2)}$ of shape $\mathbf{4B3A}$}

The algebra $A$ was constructed from an axet $X$ with $21$ axes and the automorphism group $G\cong \Aut(L_3(2))$, which contains $G_0=Miy(A)\cong L_2(7)\cong L_3(2)$ as an index two subgroup.

\begin{computation}\label{prelim L3_2 49} 
~
\begin{enumerate}
\item $A$ has no Jordan axes;
\item $A$ contains no twins for axes from $X$, and hence the tau map is a bijection between $X$ and the set of $21$ involutions from $G_0$;
\item there are no axes in $A$ corresponding to the $28$ involutions in $G\setminus G_0$.
\end{enumerate}
\end{computation}

We choose as $Y=\{a_1,a_2,a_3\}$ a triple of axes corresponding to involutions in a subgroup $E\cong 2^2$ in $G_0$. Looking at the shape, the three axes $a_i$ generate a $3$-dimensional subalgebra $V$ isomorphic to the Norton-Sakuma algebra $2\A$. In particular, in this case the axes $a_i$ do not associate with each other, and so we should not expect a complete decomposition of $A$ into joint eigenspaces.

Indeed, we have the following.

\begin{computation}
The joint eigenspaces in $A$ are as follows:
\begin{enumerate}
\item \(U:=A_{(0,0,0)}(Y)\) is of dimension $10$;
\item the remaining non-zero summands $W=A_{(\lm_1,\lm_2,\lm_3)}(Y)$, are:
\begin{enumerate}
\item \(A_{\frac{1}{4},\left(\frac{1}{32},\frac{1}{32}\right)}(Y)\),
 \(A_{\left(\frac{1}{32},\frac{1}{4},\frac{1}{32}\right)}(Y)\), and
 \(A_{\left(\frac{1}{32},\frac{1}{32},\frac{1}{4}\right)}(Y)\), of dimension $4$;
\item \(A_{\left(0,\frac{1}{32},\frac{1}{32}\right)}(Y)\), \(A_{\left(\frac{1}{32},0,\frac{1}{32}\right)}(Y)\), and \(A_{\left(\frac{1}{32},\frac{1}{32},0\right)}(Y)\), of dimension $6$.
\end{enumerate} 
\end{enumerate}  
\end{computation}

Hence, 
$$A^\circ=\bigoplus_{(\lm_1,\lm_2,\lm_3)\in \left\{1,0,\frac{1}{4},\frac{1}{32}\right\}^3}A_{(\lm_1,\lm_2,\lm_3)}(Y)$$ 
is of dimension $40$. Adding $V$, we obtain $V\oplus A^\circ$ of dimension $43$, and so it is 
still not the entire algebra $A$. Since in this algebra $A$, as in almost all known examples, the Frobenius form is positive definite, we have that $A=A^\circ\oplus A^\sharp$, as introduced above. We note that $V\subseteq A^\sharp$, as $V$ is orthogonal to $A^\circ$. We can further decompose $A^\sharp$ as $V\oplus R$, where $R$ is the orthogonal complement of $V$ in $A^\sharp$. Note that 
$R$ is also a $U$-module.

\begin{lemma}
The orthogonal complement $R$ of $V$ within $A^\sharp$ is a $U$-module, that is, it satisfies $UR\subseteq R$.
\end{lemma}

\begin{proof}
Let $r\in R$ and $u\in U$. Then $(ur,a_i)=(ru,a_i)=(r,ua_i)=(r,0)=0$. So $ur\in A^\sharp$ (since we already know that $A^\sharp$ is a $U$-module) is in the orthogonal complement of $\la a_1,a_2,a_3\ra=V$. In other words, $ur\in R$.
\end{proof}

Thus, we can work with the following complete decomposition:
$$A=R\oplus V\oplus A^\circ=R\oplus V\oplus\left(\bigoplus_{(\lm_1,\lm_2,\lm_3)\in \left\{1,0,\frac{1}{4},\frac{1}{32}\right\}^3}A_{(\lm_1,\lm_2,\lm_3)}(Y)\right),$$
where $U=A_{(0,0,0)}(Y)$ is one of the summands in the last term.

We will start by discussing the full automorphism group of $U$. The entire variety of idempotents in $U$ is $1$-dimensional, and hence difficult to analyse. We start by noting that $N=(G_0)_Y=N_{G_0}(E)\cong S_4$ acts on $U$ and induces on it the group $\bar N=N/E\cong S_3$. (Indeed, the elements of $E$ clearly act as identity on $U$.)

We first look at the fixed subalgebra $F$ of $\bar N$ in $U$. 

\begin{computation}\label{autU S_3 L3_2 49}
~
\begin{enumerate}
\item The subalgebra $F$ has dimension $4$ and it contains exactly $12$ idempotents over the algebraic closure $\bar\Q$;
\item only one of these idempotents has length $\frac{34}{5}$; call this idempotent $d$;
\item $\ad_d$ is semi-simple on $U$ with eigenvalues $1$, $0$, $\frac{9}{10}$, and $\frac{1}{2}$ on $U$, with multiplicities $1$, $4$, $2$, and $3$, respectively;
\item $d$ satisfies a ``nearly Monster'' fusion law on $U$, that is Seress and $C_2$-graded, with $U_{\frac{1}{2}}(d)$ being the negative part of $U$.
\end{enumerate}
\end{computation}

This means that $U$ has an additional automorphism $\tau_d$, which centralises $\bar N$. So now we have a subgroup isomorphic to $2\times S_3$ of $\Aut(U)$. We aim to show that this is in fact the entire group $\Aut(U)$. Note that the variety of idempotents of length $\frac{34}{5}$ in $U$ is still of dimension $1$ and so we cannot be sure at this point that $d$ is invariant under $\Aut(U)$.

Let $T$ be the $0$-eigenspace of $\ad_d$. Since the fusion law for $d$ on $U$ is Seress, $T$ is a subalgebra of dimension $4$.

\begin{computation}
~
\begin{enumerate}
\item $T$ contains exactly $16$ idempotents over $\bar\Q$.
\end{enumerate}
\end{computation}

Among these idempotents we find three of length $\frac{7}{5}$ and we now focus our attention on this length. It turns out that the variety of idempotents of length $\frac{7}{5}$ in $U$ is $0$-dimensional.

\begin{computation}\label{autU L3_2 49}
~
~
\begin{enumerate}
\item $U$ contains exactly three idempotents $u_1$, $u_2$, and $u_3$ of length $\frac{7}{5}$;
\item they generate the subalgebra $T$;
\item $\ad_{u_i}$ is semi-simple on $U$ with eigenvalues $1$, $0$, $\frac{3}{10}$, and $\frac{1}{20}$, with multiplicities $1$, $4$, $2$, and $3$, respectively;
\item $u_i$ satisfies on $U$ a ``nearly Monster'' fusion law that is Seress and $C_2$-graded, with $U_{\frac{1}{20}}(u_i)$ being the negative part;
\item the subgroup $H=\la\tau_{u_1},\tau_{u_2},\tau_{u_3}\ra\leq\Aut(U)$ is isomorphic to $S_3$ and it permutes the $u_i$ transitively.
\end{enumerate}
\end{computation}

According to Computation \ref{autU L3_2 49} (a), $\Aut(U)$ acts on the set $\{u_1,u_2,u_3\}$ and, furthermore, it induces on it the full symmetric group $S_3$, since the known subgroup $H\leq\Aut(U)$ already does so. (Also, see (e).) Hence, we just need to identify the kernel of the action, \ie the triple stabiliser $\Aut(U)_{u_1,u_2,u_3}$. This is done via an additional calculation.

We note that $d$ happens to be equal to $\one_U-\one_T$. Clearly, this now means that the entire $\Aut(U)$ fixes $d$.

\begin{computation}\label{kern autU L3_2 49}
~
\begin{enumerate}
\item the identity automorphism of $T=U_0(d)$ has a $1$-dimensional space of extensions to $U_{\frac{1}{2}}(d)$;
\item the square of a random element from $U_{\frac{1}{2}}(d)$ has a non-zero projection to $U_0(d)$;
\item $U=\lla U_{\frac{1}{2}}(d)\rra$.
\end{enumerate}
\end{computation}

Based on the above computational results, we can now identify $\Aut(U)$. We let $\hat H:=\la\tau_d,\bar H\ra\cong 2\times S_3$.

\begin{proposition}
We have that $\Aut(U)=\hat H\cong 2\times S_3$.
\end{proposition}

\begin{proof}
Suppose $\phi\in\Aut(U)$ fixes all three $u_i$. Then, by Computation \ref{autU L3_2 49} (b), $\phi$ acts trivially on $T=U_0(d)=\lla u_1,u_2,u_3\rra$. Clearly, $\phi$ also fixes $d$.

According to Computation \ref{kern autU L3_2 49} (a), $\phi$ acts on $W=U_{\frac{1}{2}}(d)$ as a scalar $\lm$. Then $(w^2)^\phi=\lm^2w^2$ for each $w\in W$, and so, by (b), we must have $\lm^2=1$, that is, $\lm=\pm 1$. By (c), the value of $\lm$ identifies $\phi$, and so we see that $\lm=1$ means that $\phi=1$ and $\lm=-1$ means that $\phi=\tau_d$. In either case, $\phi\in\hat H$, which means that $\Aut(U)=\hat H$.
\end{proof}

Now that $\Aut(U)$ is known, we can proceed with the determination of $\Aut(A)$. Our next goal is the subgroup $\Aut(A)_{a_1,a_2,a_3}$. Namely, we want to prove that $\Aut(A)_{a_1,a_2,a_3}=G_{a_1,a_2,a_3}=E$. We consider an element $\phi\in\Aut(A)_{a_1,a_2,a_3}$. Then $\phi$ fixes the three axes $a_i$ and so, on the one hand, it acts trivially on $V=\la a_1,a_2,a_3\ra$ and, on the other hand, it leaves every joint eigenspace $A_{(\lm_1,\lm_2,\lm_3)}(Y)$ invariant, including $U=A_{(0,0,0)}(Y)$. Hence we will start with an automorphism of $U$ and see which ones can be extended to other summands and to the whole $A$. 

We let $W_1=A_{(0,\frac{1}{32},\frac{1}{32})}(Y)$, $W_2=A_{(\frac{1}{32},0,\frac{1}{32})}(Y)$, and $W_3=A_{(\frac{1}{32},\frac{1}{32},0)}(Y)$. Then we have the following.

\begin{computation}\label{L3_2 49 extension U to A}
~
\begin{enumerate}
\item The identity automorphism of $U$ has a $1$-dimensional space of extensions to $W_i$;
\item for randomly selected elements $w_i\in W_i$ and $u\in U$, we have that
\begin{enumerate}
\item $(w_i^2,u)\neq 0$;
\item $(w_1w_2,w_3)\neq 0$;
\end{enumerate}
\item $\lla W_1,W_2,W_3\rra=A$; and finally,
\item every non-identity automorphism of $U$ does not extend to at least one $W_i$.
\end{enumerate}
\end{computation}

Based on these, we now identify $\Aut(A)_{a_1,a_2,a_3}$.

\begin{proposition} \label{kernel}
We have that $\Aut(A)_{a_1,a_2,a_3}=E$.
\end{proposition}

\begin{proof}
Let $\hat E:=\Aut(A)_{a_1,a_2,a_3}$. If $\phi\in\hat E$ then, by Computation \ref{L3_2 49 extension U to A} (d) above, $\phi$ restricts to $U$ as the identity automorphism, since the restriction must clearly be extendable to every joint eigenspace, including all $W_i$. So we now assume that $\phi$ fixes $U$ element-wise. By (a), $\phi$ acts on each $W_i$ as a scalar $\mu_i$. Furthermore, by (b)(i), $0\neq (w_i^2,u)=((w_i^\phi)^2,u^\phi)=(\mu_i^2w_i^2,u)=\mu_i^2(w_i^2,u)$, which yields $\mu_i^2=1$. That is, $\mu_i=\pm 1$.

From (b)(ii), $0\neq (w_1w_2,w_3)=(w_1^\phi w_2^\phi,w_3^\phi)=(\mu_1\mu_2 w_1w_2,\mu_3 w_3)=\mu_1\mu_2\mu_3(w_1w_2,w_3)$, and so $\mu_1\mu_2\mu_3=1$. By (c), the values $\mu_1$, $\mu_2$, and $\mu_3$ identify $\phi$ on the entire $A$. Thus, we have no more that four different elements in $\hat E$ corresponding to the triples $(\mu_1,\mu_2,\mu_3)=(1,1,1)$, $(1,-1,-1)$, $(-1,1,-1)$, and $(-1,-1,1)$. We have four such elements in $E$, and so $\hat E=E$.
\end{proof}

This also yields the set-wise stabiliser $\hat N:=\Aut(A)_Y$.

\begin{corollary} \label{normaliser of E}
We have that $\hat N=N\cong S_4$.
\end{corollary}

\begin{proof}
According to Proposition \ref{kernel}, $\hat N$ and $N$ have the same kernel, $E$, in their 
action on $U$. Since $N$ induces the full automorphism group $\hat H\cong 2\times S_3$ on $U$, 
we must have that $\hat N=N$. 
\end{proof}

From here we switch entirely to the group theoretic arguments. We follow the same strategy as in the proof of Proposition \ref{S5} and split the proof into a series of lemmas. Let $Q$ be a minimal normal subgroup of $\hat G:=\Aut(A)$. 

We first assume that $Q$ is an elementary abelian $p$-group for some prime $p$. We can view $Q$ as an $\F_p$-module for $\Aut(A)$. Consider a subgroup $R\leq Q$ invariant under $G_0\cong L_3(2)$ and minimal subject to this condition. Then $R$ is an irreducible $G_0$-module. Note that $C_R(E)=1$. Indeed, by computation (b), $C_{\hat G}(E)=\Aut(A)_{a_1,a_2,a_3}$ and, by Proposition \ref{kernel}, we have that $C_{\hat G}(E)=E$. Since $R\cap E\leq Q\cap G_0=1$, as $G_0$ is a simple group, we do indeed conclude that $C_R(E)=1$.

\begin{lemma}
If $Q$ is an elementary abelian $p$-group then $p\in\{3,5\}$.
\end{lemma}

\begin{proof}
First of all, $p\neq 2$, since $C_R(E)=1$. Now the argument from Lemma \ref{6-trans} works without change.
\end{proof}

Next we consider these two values of $p$ individually. Notice that we can use the property from Lemma \ref{mult of three}, as its proof also applies without change.

\begin{lemma}
We have that $p\neq 5$.
\end{lemma}

\begin{proof}
Since $5$ does not divide the order of $G_0\cong L_3(2)$, the Brauer character table of $G_0$ modulo $5$ is the same as its complex character table. By Lemma \ref{mult of three}, the dimension of $R$ has to be a multiple of $3$, say, $3k$, and additionally, the value of the character on the class of involutions must be $-k$. This only leaves the two $3$-dimensional modules, dual to each other, as candidates for $R$.

So suppose that $R$ has dimension $3$. Take involutions $x,y\in G_0$, such that 
$xy$ has order $3$. Then $x$ and $y$ have a unique common $1$-space in $R$ that they both invert. Say, this is $\la r\ra$. Then, clearly, $t:=xy$ must leave $\la r\ra$ invariant, and this means that $t$ commutes with $r$, because an element of order $3$ cannot act non-trivially on a group of order $5$. Consider $y':=y^{r^3}$. Then $xy'=xr^{-3}yr^3=xyr^3r^3=tr$. Since the element $tr$ is of order $15$, we get a contradiction, since $x$, $y$ and $y'$ are all in a $6$-transposition class.
\end{proof}

\begin{lemma}
We have that $p\neq 3$.
\end{lemma}

\begin{proof}
The Brauer character table of $G_0$ modulo $3$ is available, say in GAP. By Lemma \ref{mult of three}, the dimension of $R$ has to be a multiple of $3$, say, $3k$, and additionally, the value of the character on the class of involutions must be $-k$. This only leaves the two $3$-dimensional modules, dual to each other, as candidates for $R$. (Both of these modules are only realisable over $\F_9$, so we must have that $R\cong 3^6$.)

Take involutions $x,y\in G_0$, such that $xy$ has order $4$. Then $x$ and $y$ have a common $1$-space in $R$ that they both invert, since the $-1$-eigenspace of both $x$ and $y$ have dimension more than half of $\dim(R)$. Say, this is $\la r\ra$. Then, clearly, $t:=xy$ must commute with $r$. Consider $y':=y^{r^2}$. Then $xy'=xr^{-2}yr^2=xyr^2r^2=tr$. Since the element $tr$ is of order $12$, we get a contradiction, since $x$, $y$ and $y'$ are all in a $6$-transposition class.
\end{proof}

The remaining part of the proof is also similar to the case of the $46$-dimensional algebra for $A_5$. First of all, since we ruled out all possible primes $p$, $Q$ cannot be abelian, and so $Q=L\times L\times\cdots\times $ is a direct sum of several copies of a non-abelian simple group $L$.

Let $S$ be a Sylow $2$-subgroup of $G$ and $S_0=S\cap G_0\cong D_8$ be a Sylow $2$-subgroup of $G_0$. Without loss of generalisation, we may assume that $E\leq S_0$.

\begin{lemma}
We have that $S$ is a Sylow $2$-subgroup of $\hat G=\Aut(A)$.
\end{lemma}

\begin{proof}
Let $\hat S$ be a Sylow $2$-subgroup of $\hat G$ containing $S$ and assume by contradiction that $S<\hat S$. Then $S<N_{\hat S}(S)$. Let $t\in N_{\hat S}(S)$. By Computation \ref{prelim L3_2 49} (c), there are no axes corresponding to the involutions in $S\setminus S_0$, which means that $S_0$ is invariant under conjugation by $t$. We note that $S_0\cong D_8$ contains exactly two Klein four-subgroups, $E$ and a second one, $E'$. Furthermore, $E$ and $E'$ are conjugate in $S$. Hence, correcting $t$ with a factor from $S$ if necessary, we may assume that $t$ normalises $E$. However, this means that $t$ lies in the set-wise stabiliser of $Y$, which by Corollary \ref{normaliser of E} coincides with $N$, that is, it is contained in $G_0$. This is a contradiction with the choice of $t$. Thus, $\hat S=S$.
\end{proof}

Again, no non-abelian simple group can have a Sylow $2$-subgroup of order less that $4$. If the number of factors $L$ in $Q$ is not one, then since $|S|=2^4$, we must have that $Q\cong L\times L$ contains $S$. However, this would mean that $S$ is a direct product of two (isomorphic) groups of order $4$, which would imply that $S$ is abelian, which is, clearly, not the case. Thus, we have the following.

\begin{corollary}
We have that $Q=L$ is a simple group and $G_0\leq Q$.
\end{corollary}

\begin{proof}
We just need to show the second claim. Since $S\cap Q$ is a Sylow $2$-subgroup of $Q$, we must have that $|S\cap Q|\geq 4$. On the other hand, $S_0$ has index $2$ in $S$. Hence $S_0\cap Q\neq 1$. It now follows that $G_0\cap Q\geq S_0\cap Q\neq 1$, and since $G_0\cap Q$ is normal in $G_0$, we deduce from simplicity of $G_0$ that $G_0=G_0\cap Q$, \ie $G_0\leq Q$
\end{proof}

We can now prove our final statement.

\begin{proposition}
We have that $\Aut(A)=G\cong \Aut(L_3(2))$.
\end{proposition}

\begin{proof}
First of all, $Q$ is the unique minimal normal subgroup of $\hat G$. Indeed, if there was another such subgroup $Q'$ then we would have that both $Q$ and $Q'$ are non-abelian simple groups and, at the same time, $G_0\leq Q\cap Q'$. Clearly, this implies that $Q=Q'$.

It follows that $G_{\hat G}(Q)=1$ and so $\hat G$ is an almost simple group. Recall that the involutions from $S_0$, being tau involutions of axes, are not conjugate to involutions from $S\setminus S_0$, which do not correspond to any axes by (c). By Thompson's Transfer Theorem (see \eg \cite{KS}), $\hat G$ has an index $2$ subgroup, which implies that $S_0$ is a Sylow $2$-subgroup of $Q$, since $S\leq G_0\leq Q$. Hence $Q$ is a non-abelian simple group with a dihedral Sylow $2$-subgroup $S\cong D_8$.

By \cite{GW}, $L\cong L_2(q)$, $q\equiv 7,9\mod 16$ or $L\cong A_7$. Let us consider these possibilities in turn. First suppose that $Q\cong L_2(q)$ for $q\equiv 7,9\mod 16$. Then, first of all, all involutions in $Q$ are conjugate, so they are all tau involution of axes, and hence they must form a class of $6$-transpositions. On the other hand, $L_2(q)$ contains a dihedral subgroup of order $q+1$, which means that $\frac{q+1}{2}\leq 6$, \ie $q\leq 11$. Clearly, this forces $q=7$ and $Q=G_0\cong L_2(7)$, since $Q$ contains $G_0$.

It remains to eliminate the case $Q\cong A_7$. In this case, $\Aut(Q)\cong S_7$, which means that $\hat G$ can only be isomorphic to $S_7$, since $Q$ has index at least $2$ in $\hat G$. However, $S_7$ does not contain a subgroup isomorphic to $G\cong \Aut(L_3(2))$, as the latter does not have a transitive action on seven points. This contradiction completes the proof.
\end{proof}

\subsection{The $\mathbf{57}$-dimensional algebra of shape $\mathbf{4A3C}$ for $\mathbf{L_3(2)}$}

We now give another example of an algebra whose decomposition relative to a subgroup $2^2$ is the entire algebra.

In this subsection, the algebra $A$ under consideration was constructed from an axet $X$ with 
$|X|=21$ for the group $G=\Aut(L_3(2))$. The Miyamoto group $G_0=\Miy(G)\cong L_3(2)$ is a 
subgroup of $G$ of index $2$. We begin by computing some basic data about the algebra $A$.

\begin{computation} \label{no twins 57}
~
\begin{enumerate}
\item $A$ has no Jordan axes;
\item axes from $X$ have no twins, and thus, the tau map is a bijection from $X$ onto the set 
of all $21$ involutions in $G_0$;
\item the $28$ involutions in $G\setminus G_0$ are not tau involutions of axes from $A$.
\end{enumerate}
\end{computation}

Based on this, we intend to show that $G$ is the full automorphism group of $A$ and, 
correspondingly, $X$ contains all axes from $A$.

For our set $Y$, we select three axes $a_1$, $a_2$ and $a_3$ corresponding to the involutions 
in a subgroup $E\cong 2^2$ of $G_0$. Set $N=N_G(E)\cong S_4$ to be the set-wise stabiliser of 
$Y$ in $G$, and $K=C_G(E)=E$. Then $N$ induces the group $N/E\cong S_3$ on the set $Y$.  
Because of the algebra inclusion $2\B\hookrightarrow 4\A$, we see that the axes in 
$Y=\{a_1,a_2,a_3\}$ are pairwise orthogonal, so in this case we get a complete decomposition 
of $A$ into a sum of joint eigenspaces. Indeed, we have the following.

\begin{computation}
The joint eigenspace decomposition of $A$ with respect to $Y$ is as follows:
\begin{enumerate}
\item $U=A_{(0,0,0)}(Y)$ is of dimension $9$;
\item the remainder of the non-zero summands $W=A_{(\lm_1,\lm_2,\lm_3)}(Y)$ are:
\begin{enumerate}
\item $A_{(1,0,0)}(Y)=\la a_1\ra$, $A_{(0,1,0)}(Y)=\la a_2\ra$, and $A_{(0,0,1)}(Y)=\la 
a_3\ra$;
\item $A_{\left(0,\frac{1}{4},\frac{1}{4}\right)}(Y)$, 
$A_{\left(\frac{1}{4},0,\frac{1}{4}\right)}(Y)$, and 
$A_{\left(\frac{1}{4},\frac{1}{4},0\right)}(Y)$, of dimension $1$;
\item $A_{\left(\frac{1}{4},0,0\right)}(Y)$, $A_{\left(0,\frac{1}{4},0\right)}(Y)$, and 
$A_{\left(0,0,\frac{1}{4}\right)}(Y)$, of dimension $2$;
\item $A_{\left(\frac{1}{4},\frac{1}{32},\frac{1}{32}\right)}(Y)$, 
$A_{\left(\frac{1}{32},\frac{1}{4},\frac{1}{32}\right)}(Y)$, and 
$A_{\left(\frac{1}{32},\frac{1}{32},\frac{1}{4}\right)}(Y)$, of dimension $2$;
\item $A_{\left(0,\frac{1}{32},\frac{1}{32}\right)}(Y)$, 
$A_{\left(\frac{1}{32},0,\frac{1}{32}\right)}(Y)$, and 
$A_{\left(\frac{1}{32},\frac{1}{32},0\right)}(Y)$, of dimension $10$.
\end{enumerate}
\end{enumerate}
\end{computation}

We now turn to the determination of $\Aut(U)$. We need to find good idempotents in $U$ that 
we can work with. Let $T$ be the fixed subalgebra of $N$ in $U$.

\begin{computation}
The subalgebra $T$ has dimension $3$ and it contains exactly eight idempotents of lengths 
$0$, $15$, $\frac{135}{16}$, $\frac{105}{16}$, $\frac{75}{8}$, $\frac{45}{8}$, 
$\frac{360}{37}$, and $\frac{195}{37}$.
\end{computation}

We will focus on the idempotent $d\in T$ of length $\frac{45}{8}$.

\begin{computation}
~
\begin{enumerate}
\item The algebra $U$ contains exactly four idempotents of length $\frac{45}{8}$; they are $d$ 
and three further idempotents $u_1$, $u_2$, and $u_3$;
\item the $1$-eigenspace of $d$ has dimension $3$, and the $1$-eigenspace of each $u_i$ is of 
dimension $2$.
\end{enumerate}
\end{computation}

It follows from this that $d$ and the triple $\{u_1,u_2,u_3\}$ are invariant under the full 
group $\Aut(U)$.

\begin{computation}
$U=\lla u_1,u_2,u_3\rra$.
\end{computation}

We are ready to identify $\Aut(U)$.

\begin{proposition} \label{Aut(U) 57}
We have that $\Aut(U)\cong S_3$.
\end{proposition}

\begin{proof}
We have seen that the set $R:=\{u_1,u_2,u_3\}$ generates $U$ and is invariant under $\Aut(U)$. 
Consequently, $\Aut(U)$ acts on $R$ faithfully, which means that $\Aut(U)$ is isomorphic to a 
subgroup of $S_3$. On the other hand, we have seen earlier that $N$ induces on $U$ the full 
$S_3$. So the claim holds.
\end{proof}

We note that $U_1(d)$ contains three idempotents of length $2$ which satisfy a fusion law of 
Monster type $\left(\frac{4}{15},\frac{1}{15}\right)$ on $U$. Thus, we get three tau 
involutions of $U$ which generate $\Aut(U)$. In fact, these idempotents are the totality of 
idempotents of length $2$ in $U$.

We now focus on the $2$-dimensional summands 
$W_1=A_{\left(\frac{1}{4},\frac{1}{32},\frac{1}{32}\right)}(Y)$, 
$W_2=A_{\left(\frac{1}{32},\frac{1}{4},\frac{1}{32}\right)}(Y)$, and 
$W_3=A_{\left(\frac{1}{32},\frac{1}{32},\frac{1}{4}\right)}(Y)$ because of the following 
computational facts.

\begin{computation} \label{kernel 57}
The following hold:
\begin{enumerate}
\item $\lla W_1,W_2,W_3\rra=U$;
\item the identity automorphism on $U$ has a $1$-dimensional space of extensions to each $W_i$;
\item for randomly chosen $0\ne w_i\in W_i$, $i=1$, $2$, and $3$, and $0\ne u\in U$, we have 
\begin{enumerate}
\item $(w_i^2,u)\ne 0$; and furthermore,
\item $((w_1w_2)(w_1w_3),w_1)\ne 0$.
\end{enumerate}
\end{enumerate}
\end{computation}

This computation yields the following result. First we introduce some notation. Let $\hat N$ 
be the set-wise stabiliser of $Y=\{a_1,a_2,a_3\}$ in $\Aut(A)$ and $\hat{K}\unlhd\hat N$ be 
the joint stabiliser of the three axes from $Y$. Then the elements of $\hat K$ act on each 
component of the decomposition of $A$ with respect to $Y$. In particular, they act on $U$ and 
on the components $W_i$.

\begin{lemma} \label{hat K}
The group $\hat{K}$ acts on $U$ with kernel $E$.
\end{lemma}

\begin{proof}
Let $\phi\in\hat K$ and $\phi|_U=1$. By Computation \ref{kernel 57} (b), $\phi$ acts as a 
scalar $\mu_i$ on each $W_i$. We note that, by Computation \ref{kernel 57} (a), any 
automorphism of $A$ is completely determined by its action on the components $W_i$. By 
part (c)(i) of the same computation, we have that $0\ne(w_i^2,u)=((w_i^2)^\phi,u^\phi)
=((w_i^\phi)^2,u)=((\mu_iw_i)^2,u)=(\mu_i^2w_i^2,u)=\mu_i^2(w_i,u)$. Thus, $\mu_i^2=1$ and 
so $\mu_i=\pm 1$ for each $i$. Now, by part (c)(ii), we have
\begin{eqnarray*}
((w_1w_2)(w_1w_3),w_1)&=&((w_1w_2)^\phi(w_1w_3)^\phi,w_1^\phi)=\mu_1((w_1^\phi w_2^\phi)(w_1^\phi w_3^\phi), w_1)\\
	&=&\mu_1((\mu_1w_1\mu_2w_2)(\mu_1w_1\mu_3w_3),w_3)\\
	&=&\mu_1^2\mu_1\mu_2\mu_3((w_1w_2)(w_1w_3),w_1)\\
	&=& \mu_1\mu_2\mu_3((w_1w_2)(w_1w_3),w_1),
\end{eqnarray*}
since $\mu_1^2=1$. It follows that $\mu_1\mu_2\mu_3=1$. These two constraints imply that 
exactly two of the three $\mu_i$ can be negative for an automorphism $\phi$. Thus, 
$(\mu_1,\mu_2,\mu_3)$ lies in
$$\{(1,1,1),(1,-1,-1),(-1,1,-1),(-1,-1,1)\}.$$

These tuples correspond to the identity automorphism, $\tau_1$, $\tau_2$, and $\tau_3$, 
respectively, where $\tau_i=\tau_{a_i}$. Thus, $\phi\in E$ in all cases.
\end{proof}

Note that $(w_1w_2,w_3)$ is zero for random $w_i\in W_i$. This is why we used a longer product 
in Computation \ref{kernel 57} (c)(ii).

In order to determine the subgroups $\hat K$ and $\hat N$ of $\Aut(A)$ we will need an 
additional computation. Let $T_i$ be the subspace of $U$ spanned by the projections to $U$ of 
all products of elements of $W_i$, $i=1,2,3$.

\begin{computation}
The subspaces $T_1$, $T_2$ and $T_3$ of $U$ have dimension $3$ and any two of them intersect 
trivially.
\end{computation}

The important part of this statement is that the subspaces $T_i$ are not the same. This gives 
us the following.

\begin{lemma}\label{cent E is 1}
$\hat K$ acts trivially on $U$; in particular, $\hat K=E$.
\end{lemma}

\begin{proof}
Clearly $\hat N/\hat K\cong S_3$. Let $R$ be the kernel of $\hat N$ acting on $U$. (We note that the entire $\hat N$ preserves $U$ and acts on it.) By the preceding computation, every element of $R$ distinguishes components $W_i$, since it fixes the corresponding projection 
subspaces $T_i$. Consequently, it also distinguishes the axes $a_i$, \ie $R\leq \hat K$. 
On the other hand, $\hat N/R$ is isomorphic to a subgroup of $\Aut(U)\cong S_3$. This shows 
that, in fact, $R=\hat K$. 

Finally, the kernel of $\hat K$ acting on $U$ is simply $\hat K\cap R=\hat K$, and so Lemma 
\ref{hat K} gives us that $\hat K=E$.
\end{proof}

Finally, we can identify $\hat N$.

\begin{corollary} \label{N 57}
$\hat N=N\cong S_4$.
\end{corollary}

\begin{proof}
Indeed, $N\leq\hat N$. Also, $N$ induces the same group $S_3$ on $Y$ and it has the same 
kernel $E$ in this action.
\end{proof}

We now use group theoretic arguments to show that the full automorphism group ${\hat 
G}:=\Aut(A)$ is isomorphic to $\Aut(L_3(2))\cong PGL(2,7)$. As in the previous cases, we 
will split the proof into a series of lemmas. First, we note that ${\hat G}$ is a finite 
group by Corollary \ref{our main case}. 

Let $1\neq Q$ be a minimal normal subgroup of $\hat G$. We  begin by showing that the soluble 
radical of ${\hat G}$ is trivial. For a contradiction, we will be assuming in the next several 
lemmas that $Q$ is an elementary abelian $p$-group for some prime $p$. Then $Q$ can be 
regarded as a vector space over $\F_p$, and also, as a $\hat{G}$-module. Let $R$ be 
a smallest non-trivial subgroup of $Q$ invariant under $G_0=L_3(2)$. Then $R$ is an 
irreducible $G_0$-module. 

\begin{lemma}\label{centraliser E in R}
We have $C_R(E)=1$.
\end{lemma}

\begin{proof}
By Lemma \ref{cent E is 1}, $C_{\hat{G}}(E)={\hat G}_{a_1,a_2,a_3}=\hat K=E$, and since 
$C_R(E)=R\cap C_{\hat{G}}(E)$, we have that $C_R(E)=R\cap E\leq R\cap G_0\unlhd G_0$. 
Simplicity of $G_0$ then gives that $C_R(E)=1$.
\end{proof}

In particular, $R$ cannot be the trivial $G_0$-module. 

\begin{lemma}
If $Q$ is an elementary abelian $p$-subgroup of ${\hat G}$, then $p\in\{3,5\}$.
\end{lemma}

\begin{proof}
If $p=2$, then $C_R(E)\ne 1$, since $E$ is a $2$-group. This contradicts Lemma 
\ref{centraliser E in R}, so $p\ne 2$. Again the arguments from \ref{6-trans}, utilising the 
$6$-transposition property of tau involutions, apply without change, so $p\in \{3,5\}$ as 
claimed.
\end{proof} 

We now proceed to eliminate the cases $p=3$ and $p=5$ individually.

\begin{lemma}
$p\ne 5$.
\end{lemma}

\begin{proof}
Since $5$ is relatively prime to $|G_0|=168$, the $5$-modular character table of $G_0$ 
coincides with its ordinary character table. By Lemma \ref{mult of three}, the dimension of 
$R$ is of the form $3k$, $k \in \N$. In addition, the value of the character afforded by 
$R$ must be $-k$ on the class of involutions of $G_0$. This leaves us with the two mutually 
dual $3$-dimensional modules as candidates for $R$. 

Both of these modules can only be realised over the extension $\F_{5^6}$ of $\F_5$, so in this 
case $R$ would be of dimension $18$ and $F:=C_{\End(R)}(G_0)\cong\F_{5^6}$. So we can consider 
$R$ as the $FG_0$-module of dimension $3$. Select involutions $x$ and $y$ of $G_0$ such that 
$t:=xy$ is of order $3$. Since both $x$ and $y$ have $-1$-eigenspaces of dimension 
$2>\frac{3}{2}$, they have a $1$-dimensional space they both invert, say $\la r\ra$. Clearly, 
$t$ centralises $r$, and since they have coprime orders, $tr$ has order $15$. On the other 
hand, $x$ inverts $tr$ and so it cannot be a $6$-transposition; a contradiction.
\end{proof}

\begin{lemma}
We have $p\ne 3$.
\end{lemma}

\begin{proof}
As in the preceding lemma, $R$ has dimension $n=3k$, and the character value of the class of 
involutions must be $-k$. From the $3$-modular character table of $G_0=L_3(2)$, we are left 
with two possibilities for $R$, the two mutually dual $3$-dimensional modules. Choose 
involutions $x$ and $y$ in $G_0$ so that $t:=xy$ is of order $4$. Then $x$ and $y$ have a 
common $1$-dimensional subspace $\la r\ra\leq R$ they both invert. Then $t$ commutes with $r$. 
Let $y'=y^{r^2}$, and consider $xy'=xr^{-2}yr^2=xr^{-2}xxyr^2=r^2tr^2=tr^2r^2=tr$. Clearly, 
$tr$ is of order $12$, contradicting the fact $x$ and $y'$ are $6$-transpositions. Thus, 
$p\ne 3$.
\end{proof}

We have established that $Q$ cannot be elementary abelian, and so it must be a direct product 
of several copies of a non-abelian simple group $L$. We want to show that $Q=L$ is simple. For 
this we look at the Sylow $2$-subgroup $\hat S$ of $\hat G$. We can assume that $\hat S$ 
contains a Sylow $2$-subgroup $S$ of $G$. Let $S_0=S\cap G_0$.

\begin{lemma}
$\hat S=S$.
\end{lemma}

\begin{proof}
Let $T=N_{\hat S}(S)$. It suffices to show that $T=S$. Every element of $T$ acts on $S$. By 
Computation \ref{no twins 57}, the involutions from $S\setminus S_0$ are not fused into $S_0$, 
which means that $T$ normalises $S_0$. The latter contains exactly two Klein four-groups, $E$ 
and the second one, say, $E'$. This means that $N_T(E)$ has index at most two in $T$. However, 
$N_{\hat G}(E)=\hat N=N\leq G_0$ by Corollary \ref{N 57}. This shows that $N_T(E)\leq T\cap 
G_0=S_0$. Since $S_0$ has index two in $S$, we now conclude that $T=S$, and so the claim of 
the lemma holds.
\end{proof}

\begin{lemma}
The group $\hat G$ has an index two subgroup $\hat G_0$ containing $G_0$.
\end{lemma}

\begin{proof}
Again, we refer to the computational fact that the involutions from $S\setminus S_0$ are not 
fused into $S_0$. By Thompson's Transfer Theorem, $\hat G$ has an index $2$ subgroup $\hat 
G_0$ containing $S_0$. Since $S_0\leq G_0$ and $G_0$ is simple, we conclude that $G_0$ must be 
fully contained in $\hat G_0$. 
\end{proof}

\begin{corollary}
$Q=L$ is simple.
\end{corollary}

\begin{proof}
Clearly, $Q\leq\hat G_0$. In particular, as $S_0$ must clearly be Sylow in $\hat G_0$, the 
$2$-part of $|Q|$ is at most $2^3$. Since a non-abelian simple group has Sylow $2$-subgroups 
of order at least $4$, we are forced to conclude that $Q$ is simple.
\end{proof}

We note that $S_0\cap Q\neq 1$ and so $G_0\cap Q\neq 1$. This yields that $G_0$, being simple, 
is fully contained in $Q$. This also implies that $Q$ is the only minimal normal subgroup of 
$\hat G$, which means that $C_{\hat G}(Q)=1$ and so $\hat G$ is isomorphic to a subgroup of 
$\Aut(Q)$.

We now use the classification of finite simple groups with a dihedral Sylow subgroup to 
identify $\hat G$. 

\begin{theorem}
$Aut(A)=G\cong\Aut(L_3(2))$.
\end{theorem}

\begin{proof}
By \cite{GW}, $Q\cong L_2(q)$ for $q\equiv 7,9\mod 16$ or $Q\cong A_7$. First suppose that 
$Q\cong A_7$. Then $\hat G\cong S_7$. However, this gives a contradiction since such a group 
$\hat G$ cannot contain a subgroup $G\cong\Aut(L_3(2))$, since the latter does not have a 
transitive action on seven points. 

Now suppose that $Q\cong L_2(q)$. Since $L_2(q)$, for odd $q$, has a single class of 
involutions, we must have that this class is a $6$-transposition class, which shows that 
$q+1\leq 12$ (as $L_2(q)$ contains a dihedral subgroup of order $q+1$). Clearly, $L_2(11)$ and 
$L_2(9)$ contain no subgroup $L_3(2)\cong L_2(7)$. Thus, $Q\cong L_2(7)\cong L_3(2)$, that is, 
$Q=G_0$. This clearly yields that $\hat G=G$.
\end{proof}

\section{Two largest algebras} \label{largest}

In this section we determine the automorphism groups of two algebras for the group $S_6$.

\subsection{The $\mathbf{121}$-dimensional algebra for $\mathbf{A_6}$ of shape $\mathbf{4A3A3A}$}\label{A6_121}

The first algebra $A$ arises for the action of $G=S_6$ on the axet of $45$ even involutions 
(double $2$-cycles).\footnote{This algebra was not computed by the expansion algorithm, like 
others, but instead it was found as a subalgebra of a larger algebra, that we discuss in the 
next subsection. In particular, it is not known whether this algebra is the unique algebra of 
its shape.} Therefore, the Miyamoto group in this case is $G_0=A_6$. As usual, we first check 
for Jordan elements, twins and axes for the additional involutions from $G$.

We note that even involutions form a single conjugacy class in both $A_6$ and $S_6$, and 
$S_6$ has two additional conjugacy classes of involutions in $S_6\setminus A_6$, namely, 
$2$-cycles and triple $2$-cycles. 

\begin{computation} \label{G}
~
\begin{enumerate}
\setcounter{enumi}{0}
\item $A$ has no Jordan axes;
\item the $45$ known axes do not have twins;
\item involutions in the two conjugacy classes contained in $G\setminus G_0$ do not correspond to axes; and
\item the outer automorphisms of $S_6$ induce automorphisms of $A$, so that the known automorphism group of $A$ is now $G^\circ\cong\Aut(A_6)$.
\end{enumerate}
\end{computation}

Inside $G^\circ$, the two classes of involutions from $S_6\setminus A_6$ merge into a single conjugacy class. However, we also find an additional class of $36$ involutions in $G^\circ\setminus G$.

\begin{computation} \label{G circ}
Involutions in $G^\circ\setminus G$ do not correspond to axes in $A$.
\end{computation}

Hence our task is now clear: we aim to prove that $\Aut(A)=G^\circ\cong\Aut(A_6)$ and that $A$ 
contains exactly $45$ axes of Monster type $(\frac{1}{4},\frac{1}{32})$.

Turning to the decomposition of $A$, we select $Y=\{a_1,a_2,a_3\}$ to consist of three axes 
corresponding to three involutions in a subgroup $E\cong 2^2$ of $G_0\cong A_6$. Since $4\A$ in 
the shape contains subalgebras $2\B$, the three axes $a_1$, $a_2$, and $a_3$ are pairwise 
annihilating, and so we are in the easier case, where we immediately obtain a complete 
decomposition of $A$.

\begin{computation}
The joint eigenspace decomposition of $A$ corresponding to $Y$ is as follows:
\begin{enumerate}
\item $U:=A_{(0,0,0)}(Y)$ is of dimension $19$;
\item the remaining (non-zero) summands $A_{(\lm_1,\lm_2,\lm_3)}(Y)$ in the decomposition of 
$A$ are:
\begin{enumerate}
\item $A_{(1,0,0)}(Y)=\la a_1\ra$, $A_{(0,1,0)}(Y)=\la a_2\ra$, and $A_{(0,0,1)}(Y)=\la a_3\ra$;
\item $A_{\left(0,\frac{1}{4},\frac{1}{4}\right)}(Y)$, 
$A_{\left(\frac{1}{4},0,\frac{1}{4}\right)}(Y)$, and 
$A_{\left(\frac{1}{4},\frac{1}{4},0\right)}(Y)$ of dimension $2$;
\item $A_{\left(0,0,\frac{1}{4}\right)}(Y)$, $A_{\left(0,\frac{1}{4},0\right)}(Y)$, and 
$A_{\left(\frac{1}{4},0,0\right)}(Y)$, of dimension $5$;
\item $A_{\left(\frac{1}{4},\frac{1}{32},\frac{1}{32}\right)}(Y)$, 
$A_{\left(\frac{1}{32},\frac{1}{4},\frac{1}{32}\right)}(Y)$, and 
$A_{\left(\frac{1}{32},\frac{1}{32},\frac{1}{4}\right)}(Y)$, of dimension $6$;
\item $A_{\left(0,\frac{1}{32},\frac{1}{32}\right)}(Y)$, 
$A_{\left(\frac{1}{32},0,\frac{1}{32}\right)}(Y)$, and 
$A_{\left(\frac{1}{32},\frac{1}{32},0\right)}(Y)$, of dimension $20$.
\end{enumerate}
\end{enumerate}
\end{computation}

As usual, we start by finding $\Aut(U)$, and since it is quite big, we want to decompose it 
further. First, let us find at least some interesting idempotents in $U$. We start with the 
$2$-dimensional components $W_1=A_{\left(0,\frac{1}{4},\frac{1}{4}\right)}(Y)$, 
$W_2=A_{\left(\frac{1}{4},0,\frac{1}{4}\right)}(Y)$, and 
$W_3=A_{\left(\frac{1}{4},\frac{1}{4},0\right)}(Y)$. For each $i$, let $U_i$ be the subalgebra 
generated by the projections onto $U$ of all products $ww'$, $w,w'\in W_i$. Clearly, it 
suffices to take $w,w'$ in a basis of $W_i$. 

\begin{computation}
~
\begin{enumerate}
\item $U_i$ has dimension $4$; and
\item the identity $z_i$ of $U_i$ has length $4$.
\end{enumerate}
\end{computation}

Next we do a global calculation in $U$ of all idempotents of length $4$.

\begin{computation}
The idempotents $z_1$, $z_2$ and $z_3$ are the only idempotents in $U$ of length $4$.
\end{computation}

Clearly, the set-wise stabiliser of $Y$ coincides with the normaliser $N=N_{G^\circ}(E)\cong 
2\times S_4$, and it induces the full permutation group $S_3$ on $Y$. Consequently, it also 
induces $S_3$ on the set of components $\{W_1,W_2,W_3\}$, on the set of projection subalgebras 
$\{U_1,U_2,U_3\}$, and finally, on the set of the identity elements, $\{z_1,z_2,z_3\}$. So, 
in order to find the full automorphism group of $U$, we just need to find the joint stabiliser 
$K$ of $z_1$, $z_2$ and $z_3$ in $\Aut(U)$.

Let $V$ be the subalgebra of $U$ generated by $z_1$, $z_2$, and $z_3$.

\begin{computation}
~
\begin{enumerate}
\item $V$ is of dimension $4$;
\item for $i\in\{1,2,3\}$, $z_i$ is not primitive in $V$; namely, $V_1(z_i)$ is of dimension 
$2$;
\item each $V_1(z_i)$ contains exactly four idempotents, $0$, $z_i$, and two idempotents of 
length $2$, $u$ and $u_i$;
\item $uu_i=0$ and $u+u_i=z_i$.
\end{enumerate}
\end{computation}

Note that the idempotent $u$ is common for all $z_i$. In particular, it is fixed by the entire 
$\Aut(U)$. We now focus on this idempotent $u$ and decompose $U$ with respect to it.

\begin{computation} \label{u}
~
\begin{enumerate}
\item The idempotent $u$ has spectrum $1,0,\frac{1}{2},\frac{1}{8}$ on $U$, with 
multiplicities $1$, $10$, $3$ and $5$, respectively;
\item $u$ satisfies the fusion law of ``almost" Monster type $\left(\frac{1}{2},
\frac{1}{8}\right)$ on $U$.
\end{enumerate}
\end{computation}

Let $W$ be the $5$-dimensional eigenspace $U_{\frac{1}{8}}(u)$. Since $K$ fixes $u$, it 
leaves $W$ invariant. Furthermore, the idempotents $u_i$ are in $U_0(u)$ and hence they also 
act on $W$, as $0\star\frac{1}{8}=\{\frac{1}{8}\}$ in the fusion law for $u$.

\begin{computation}
In its action on $W$, the idempotent $u_i$ has eigenvalues $\frac{25}{168}$, 
$\frac{1}{168}$, $\frac{3}{56}$, and $\frac{7}{24}$ with multiplicities $1$, 
$2$, $1$, and $1$, respectively.
\end{computation} 

Let us focus on the $1$-dimensional eigenspaces $T_i=W_{\frac{7}{24}}(u_i)$. 

\begin{computation}
The Frobenius form is non-zero on $T_i$.
\end{computation}

After scaling random elements from $T_i$, we end up with $t_i\in T_i$ such that $(t_i,t_i)=21$. Note that such $t_i$ are unique up to a factor $\pm 1$.

\begin{computation} \label{ti}
~
\begin{enumerate}
\item $(t_i,t_j)=\pm\frac{15}{8}$ for $i\neq j$.
\item $\lla t_1,t_2,t_3\rra=U$.
\end{enumerate}
\end{computation}

Consider $\phi\in K$. Since $\phi$ fixes the idempotents $u_i$, it leaves all $T_i$ invariant. 
Furthermore, since $T_i$ is $1$-dimensional, $\phi$ acts on $T_i$ as a scalar $\mu_i$.

\begin{lemma} \label{K}
$K$ has order $2$.
\end{lemma}

\begin{proof}
We have that $21=(t_i,t_i)=(t_i^\phi,t_i^\phi)=(\mu_1 t_i,\mu_i t_i)=\mu_i^2(t_i,t_i)=
\mu_i^2 21$. This implies that $\mu_i=\pm 1$. Similarly, for $i\neq j$, we have from part (a) 
of Computation \ref{ti} that $0\ne (t_i,t_j)=(t_i^\phi,t_j^\phi)=(\mu_it_i,\mu_j t_j)=
\mu_i\mu_j(t_i,t_j)$. Hence $\mu_i\mu_j=1$. This means that $\mu_i$ and $\mu_j$ must 
have the same sign. Since this holds for any pair $i$ and $j$, all $\mu_i$ are equal. This 
shows, in view of part (b) of Computation \ref{ti}, that $K$ has order at most $2$. 

On the other hand, since $u$ is of ``almost'' Monster type on $U$ (see Computation \ref{u} 
(b)), the Miyamoto involution $\tau_u$ is contained in $K$, and so $|K|=2$.
\end{proof}

\begin{corollary} \label{AutU}
$\Aut(U)\cong 2\times S_3$.
\end{corollary}

\begin{proof}
We know that $K\cong 2$ and $\Aut(U)/K\cong S_3$. On the other hand, $N$ induces on $U$ a 
subgroup $S_3$, which means that $\Aut(U)$ is a split extension $2\times S_3$.
\end{proof}

Now we assume that an automorphism $\phi$ acts trivially on $U$ and we want to see how 
$\phi$ can extend on the entire $A$. We focus on the components $W_1=A_{\left(\frac{1}{4},\frac{1}{32},\frac{1}{32}\right)}(Y)$, $W_2=A_{\left(\frac{1}{32},\frac{1}{4},\frac{1}{32}\right)}(Y)$ and $W_3=A_{\left(\frac{1}{32},\frac{1}{32},\frac{1}{4}\right)}(Y)$. 

\begin{computation}
~
\begin{enumerate}
\item the identity automorphism on $U$ has $1$-dimensional spaces of extensions on each 
$W_i$; and
\item $\lla W_1,W_2,W_3\rra =A$.
\end{enumerate}
\end{computation}  

It follows from here that $\phi$ acts as a scalar $\nu_i$ on $W_i$. Select random $w_i\in W_i$ and $u\in U$. 

\begin{computation} \label{extensions_19_to_121}
~
\begin{enumerate}
\item $(w_i^2,u)\ne 0$; and
\item $(w_1w_2,w_3)\ne 0$; and
\end{enumerate}
\end{computation}

This leads to the following result. Let $\hat N=N_{\Aut(G)}(E)$ be the set-wise stabiliser 
of $Y$ in $\Aut(A)$.

\begin{lemma}
The kernel $\hat K$ of $\hat N$ acting on $U$ is of order at most $4$.
\end{lemma}

\begin{proof}
By Computation \ref{extensions_19_to_121} (a), we have that $0\neq (w_i^2,u)=
((w_i^2)^\phi,u^\phi)=((w_i^\phi)^2,u)=((\nu_i w_i)^2,u)=\nu_i^2(w_i^2,u)$, which gives us 
that $\nu_i^2=1$, and so $\nu_i=\pm 1$. Similarly, by Computation \ref{extensions_19_to_121} 
(b), $0\neq (w_1w_2,w_3)=((w_1w_2)^\phi,w_3^\phi)=(\nu_1w_1\nu_2w_2,\nu_3w_3)=
\nu_1\nu_2\nu_3(w_1w_2,w_3)$, and so $\nu_1\nu_2\nu_3=1$. This implies the claim.
\end{proof}

\begin{corollary}
We have that $\hat K=E$ and $\hat N=N$.
\end{corollary}

\begin{proof}
Clearly $E$ acts trivially on $U$, and so $E\leq\hat K$. On the other hand, $|E|=4$ and 
$|\hat K|\leq 4$, so $\hat K=E$. This means that $N$, which is contained in $\hat N$, 
induces on $U$ the group $N/E\cong 2\times S_3$, which is isomorphic to the full automorphism 
group of $U$ by Corollary \ref{AutU}. Hence $\hat N=N$.
\end{proof}

We will now use group theoretic arguments to prove that the full automorphism group 
of $A$ is isomorphic to $\Aut(A_6)$.

\begin{theorem} \label{main 13}
$\Aut(A)=G^\circ\cong\Aut(A_6)$.
\end{theorem}

We prove this in a series of lemma. As before, our first goal is to show that the soluble radical of $\hat G=\Aut(A)$ is trivial. We let $Q$ be a minimal normal subgroup of $\hat G$.
Assume by contradiction that $Q$ is abelian, namely, it is elementary abelian $p$-group for 
some prime $p$. We regard $Q$ as a $\hat G$-module over $\F_p$.

\begin{lemma} \label{not 2 again}
We have that $p\neq 2$.
\end{lemma}

\begin{proof}
Since $E$ is a $2$-group, it has a non-trivial centraliser in $Q$. However, $C_Q(E)\leq N_{\hat G}(E)=\hat N=N$. Consequently, $Q$ intersects $N\leq G^\circ$ non-trivially. This is a contradiction since $G^\circ\cong\Aut(A_6)$ has no non-trivial soluble normal subgroups.
\end{proof}

\begin{lemma}
We have that $p=3$ or $5$.
\end{lemma}

\begin{proof}
Select $1\neq e\in E$ and note that $e$ cannot act on $Q$ as identity. Indeed, as all 
non-identity elements of $E$ are conjugate in $N$, we would then have that the entire $E$ acts 
trivially on $Q$, which would mean that $Q\leq C_{\hat G}(E)\leq N$, clearly a contradiction. 
So $e$ must have a non-trivial $-1$-eigenspace in $Q$, and taking $r$ to be a non-trivial 
element of this eigenspace, we conclude that $|ee^r|=|r^2|=p$. On the other hand, $e$, being a 
Miyamoto involution, belongs to a $6$-transposition class. Hence the conclusion of the lemma 
holds, since $p\neq 2$ by Lemma \ref{not 2 again}.
\end{proof}

Let $R$ be an irreducible $G_0$-submodule in $Q$. As in the earlier cases, we can deduce that 
the dimension of $R$ can only be $3k$ for some $k$ and the value of the character on $1\neq 
e\in E$ should be $-k$. The modular character tables of $G_0\cong A_6$ are well-known and are
available, say, in GAP.

\begin{lemma}
We have that $p\neq 5$.
\end{lemma}

\begin{proof}
The irreducible degrees of $A_6$ modulo $5$ are $1$, $5$, $8$, and $10$. None of these is a 
multiple of $3$.
\end{proof}

\begin{lemma}
Also, $p\neq 3$.
\end{lemma}

\begin{proof}
Recall that $Q$ must have dimension $3k$ and the character value on $e$ must be $-k$. Looking 
at the character table of $A_6$ modulo $3$, we see that we must have $k=1$. Let $t\in G_0$ be 
an element of order $5$ inverted by $e$. Then $R$ decomposes with respect to $\la t\ra$ as 
$C_R(t)\oplus [R,t]$, where the first summand is $1$-dimensional. As $e$ cannot act on $[R,t]$ 
as a scalar, we have that $e$ has eigenvalues $1$ and $-1$ in $[R,t]$, and hence it has the 
eigenvalue $-1$ on $C_R(t)$. Select $r\in C_R(t)$, $r\neq 1$. Then $|rt|=15$ and $e$ inverts 
$rt$, which implies that $|ee^{rt}|=|(rt^{-1})^ert|=|(rt)^2|=15$, which is a contradiction, 
since $e$ is a $6$-transposition.
\end{proof}

We now have a contradiction, which shows that the minimal normal subgroup $Q$ of $\hat G$ 
cannot be abelian. So it is a direct product of isomorphic non-abelian simple groups. Let us 
now focus on a Sylow $2$-subgroup of $\hat G$.

Let $S$ be a Sylow $2$-subgroup of $N$ and $S^\circ$ be a Sylow $2$-subgroup of $G^\circ$ 
containing $S$. Let also $S_0=S\cap G_0$. Then $S_0$ is normal in $S^\circ$ and 
$[S^\circ:S_0]=4$. Let $T$ be a Sylow $2$-subgroup of $\hat G$ containing $S^\circ$.

\begin{lemma}
We have that $T=S^\circ\leq G^\circ$.
\end{lemma}

\begin{proof} 
According to Computations \ref{G} and \ref{G circ}, the involutions in $S_0$ are tau 
involutions of axes, while none of the involutions in $S^\circ\setminus S_0$ is a tau 
involution. In particular, the involutions from $S^\circ\setminus S_0$ are not fused 
into $S_0$. 

Let $T_0=N_T(S^\circ)$. Since the involutions from $S^\circ\setminus S_0$ 
are not fused into $S_0$, we conclude that $T_0$ normalises $S_0$ (as $S_0\cong D_8$ is 
generated by its involutions). The group $S_0$ contains exactly two elementary abelian 
subgroups of order $4$, $E$ and a second one, $E'$. Hence $N_{T_0}(E)$ has index at most $2$ 
in $T_0$. On the other hand, $N_{T_0}(E)=N\cap T_0=S$, since $S\leq S^\circ\leq T_0$ and $S$ 
is a Sylow $2$-subgroup of $N$. We conclude that $S$ has index at most $2$ in $T_0$. However, 
$S$ has index $2$ in $S^\circ$ and so $S^\circ\leq T_0$. This gives us that $T_0=S^\circ$. 
Finally, since $T_0=N_T(S^\circ)=S^\circ$, we obtain that $T=S^\circ$.
\end{proof}

\begin{lemma}
We have that $G_0\leq Q$.
\end{lemma}

\begin{proof}
Suppose by contradiction that $G_0\not\leq Q$. Since $G_0\cong A_6$ is simple, we must then have that $Q\cap G_0=1$. On the other hand, the Sylow $2$-subgroup $S_0$ of $G_0$ has index $4$ in $S^\circ$. This shows that the Sylow $2$-subgroup of $Q$ is of order at most $4$. Since 
it is non-soluble, we have that $Q$ has a Sylow $2$-subgroup of order exactly $4$. Now we have 
a contradiction as follows: $S^\circ\cap Q$ is of order $4$ and normal in $S^\circ$ and 
also $S_0$ is normal in $S^\circ$. Since $S_0\cap Q=1$, we have that $S^\circ\cap Q$ 
centralises $S_0$. Also, $|(S^\circ\cap Q)S_0|=4|S_0|=|S^\circ|$, \ie $(S^\circ\cap 
Q)S_0=S^\circ$. This however, implies that $E$ is normal in $S^\circ$, which we know is not 
the case.
\end{proof}

We now clarify the structure of $Q$. Note that $C_{\hat G}(Q)=1$, because otherwise $\hat G$ 
would have a second minimal normal subgroup, which is impossible by the above. It follows that 
$\hat G$ is isomorphic to a subgroup of $\Aut(Q)$ containing $\Inn(Q)\cong Q$.

\begin{lemma}
We have that $Q$ is a simple group and the $2$-part of $|Q|$ is $8$.
\end{lemma}

\begin{proof}
The fact that the involutions from $S^\circ\setminus S_0$ do not fuse into $S_0$ implies 
by Thompson's Transfer Theorem that $\hat G$ has a normal subgroup of index $4$. Clearly, 
$Q$, being a direct product of non-abelian simple groups must be contained in this normal 
subgroup. On the other hand, $S_0\leq Q$ and the $2$-part of $|S_0|$ is 
$8=\frac{|S^\circ|}{4}$. Thus, the $2$-part of $|Q|$ is also $8$. Now it is clear, $Q$ can 
only be itself a simple group, since the order of a non-abelian simple group is always 
divisible by $4$. (And so there is no room for a second factor.) 
\end{proof}

Finally, we can establish Theorem \ref{main 13}. From the above, we have that $Q$ is a finite 
non-abelian simple group, and its Sylow $2$-subgroup is of order $8$. It also contains 
$G_0\cong A_6$, $Q$ has a Sylow $2$-subgroup isomorphic to $D_8$. According to \cite{GW}, 
$Q\cong L_2(q)$, $q\equiv 7,9\mod 16$ or $L\cong A_7$. We now use the fact that the elements 
of order $2$ from $E$ are $6$-transpositions. 

If $Q\cong A_7$ then $\hat G$ must be isomorphic to a subgroup of $\Aut(A_7)\cong S_7$. 
However, this contradicts the fact that the index of $Q$ in $\hat G$ is at least $4$.

So $Q\cong L_2(q)$. In this case, $Q$ has only one class of involutions and it contains a 
dihedral group of order $q+1$. This means that $\frac{q+1}{2}\leq 6$, \ie $q\leq 11$. 
Clearly, this only allows $q=7$ or $q=9$. However, $L_2(7)$ is too small to contain 
$G_0\cong A_6$ and $L_2(9)\cong A_6$, and so in this case $Q=G_0$. Now $\hat G$ must be 
isomorphic to a subgroup of $\Aut(A_6)$. However, $\hat G$ contains $G^\circ\cong\Aut(A_6)$. 
Thus, $\hat G=G^\circ$, as claimed. This completes the identification of $\Aut(A)$.

\subsection{The $\mathbf{151}$-dimensional algebra for ${\mathbf{S_6}}$ of shape $\mathbf{4A4A3A2A}$}

Let $A$ be the $151$-dimensional algebra for $G=S_6$, of shape $4\A4\A3\A2\A$, which was 
constructed from an initial axet $X$ consisting of $15+45$ axes. In this axet, the orbits $15$ 
and $45$ naturally correspond to the $2$-cycles and double $2$-cycles in $G$, respectively. 
We note that in this case $G_0=\Miy(A)=G$.

\begin{computation}\label{axet 151}
~
\begin{enumerate}
\item $A$ has a unique Jordan axis $d$;
\item each axis in the orbit $15$ in $X$ has a unique twin;
\item the $15$ triple $2$-cycles in $G$ are tau involutions of axes, two twinned axes per involution; 
\item axes in the orbit $45$ in $X$ have no twins; and
\item $A$ admits an automorphism $g$ of order $4$ preserving the axet $X^\circ$ of known 
$1+15+15+15+15+45$ axes and inducing an outer automorphism of $G\cong S_6$; furthermore, 
$g^2\in G$.
\end{enumerate}
\end{computation}

It immediately follows from Computation \ref{axet 151} that we can substitute the original 
axet $X$ with the larger axet $X^\circ$, and the original group $G=S_6$ with the larger group 
$G^\circ\cong 2\times\Aut(S_6)$. We note that the sigma involution $\sg$ corresponding to the 
Jordan axis $d$ is in the centre of $G^\circ$. We aim to prove that $G^\circ$ coincides with 
the full automorphism group $\Aut(A)$ and $X^\circ$ contains all axes from $A$.

We will now apply the decomposition method with respect to $Y=\{d\}$. 

\begin{computation}
~
\begin{enumerate}
\item $U:=A_{(0)}(Y)=A_0(d)$ is of dimension $121$;
\item the remaining summands in the decomposition are $A_{(1)}(Y)=A_1(d)=\la d\ra$ and $A_{\left(\frac{1}{4}\right)}(Y)=A_{\frac{1}{4}}(d)$ of dimension $29$.
\end{enumerate}
\end{computation}

We note that this $121$-dimensional subalgebra $U$ is the algebra of type $4\A3\A3\A$ we studied 
in Subsection \ref{A6_121}, \ie this is how that algebra $U$ was constructed. Hence we have 
from our results in Subsection \ref{A6_121} that $\Aut(U)\cong\Aut(A_6)$, and in particular, 
$G^\circ$ induces on $U$ its full automorphism group. We note that $G^\circ$ fixes $d$, and 
hence acts on $U$ and the remaining summands of the decomposition with respect to $d$. 
Furthermore, the orbit $45$ (but none of the remaining orbits from $X^\circ$) is contained in 
$U$ and this indeed shows that $G^\circ$ induces the group $\Aut(A_6)$ on $U$. It also gives 
us the following.

\begin{lemma}
The sigma involution $\sg$ corresponding to $d$ switches all pairs of twin axes in $X^\circ$.
\end{lemma}

As mentioned above and proved in Subsection \ref{A6_121}, only the orbit $45$ from $X^\circ$ 
is in $U$. So $\sg$ moves all twinned axes within the corresponding twin pairs.

To complete finding the full automorphism group of $A$, set $\hat{G}=\Aut(A)$. Since $d$ is 
the unique axis of Jordan type in $A$, $\hat{G}$ fixes $d$. Consequently, $\hat{G}$ acts 
trivially on $\la d\ra$, and also acts on $U$ and $W:=A_{\frac{1}{4}}(d)$. Since $G^\circ$ and 
$\hat{G}$ induce the same action on $U$, to prove that $\hat{G}=G^\circ$, it suffices to show 
that they have the same kernel in their action on $U$. Note that for $G^\circ$, the kernel $K$
coincides with $\la\sg\ra$. We let $\hat K$ denote the kernel of $\hat G$ acting on $U$.

\begin{computation}
~
\begin{enumerate}
\item The identity automorphism of $U$ admits a $1$-dimensional space of extensions on $W$, 
that is, every automorphism $\phi\in\hat K$ acts on $W$ as a scalar; and
\item a random element $w\in W$ satisfies $w^2\neq 0$.
\end{enumerate}
\end{computation}

Clearly, $w^2\in A_1(d)\oplus A_0(d)=\la d\ra\oplus U$, since $d$ is of Jordan type. 
 
\begin{lemma}
We have that $\hat K=\la\sg\ra$.
\end{lemma}

\begin{proof}
Every $\phi\in\hat K$ fixes $d$ and all of $U$. By (a) above, $\phi$ acts as some scalar $\lm$ 
on the complement $W$ to $\la d\ra\oplus U$ in $A$. Hence for the element $w$ from (b) above, 
we get $w^2=(w^2)^\phi=(w^\phi)^2=(\lm w)^2=\lm^2 w^2$. Since $w^2\neq 0$, it follows that 
$\lm^2=1$, \ie $\lm=\pm 1$.

If $\lm=1$, it is easy to see that $\phi$ is the identity automorphism of $A$, while if 
$\lm=-1$, then $\phi=\sg$. Thus, we have that $\hat{K}=\la\sg\ra=K$. 			
\end{proof}

Wee have established the following main result.

\begin{theorem}
The full automorphism group of the $151$-dimensional algebra $A$ of shape $4\A4\A3\A2\A$ is 
$G\cong 2\times\Aut(A_6)$. Additionally, $A$ contains exactly $1+15+15+15+15+45$ axes.
\end{theorem}

\end{document}